\documentclass[12pt]{article}

\usepackage{amsmath}
\usepackage{amssymb}
\usepackage{amsthm}
\usepackage{amsfonts}
\usepackage{latexsym}
\usepackage{cite}
\usepackage{psfrag}
\usepackage{epsfig}
\usepackage{graphicx}
\usepackage{mathrsfs}
\usepackage{url}
\usepackage{color}
\usepackage{cancel}
\usepackage{eufrak}
\usepackage{multirow}
\usepackage{hyperref}

\makeatletter
\@addtoreset{equation}{section}
\makeatother

\newtheorem{theorem}{Theorem}[section]
\newtheorem{lemma}[theorem]{Lemma}

\newtheorem{conjecture}[theorem]{Conjecture}

\theoremstyle{definition}

\newtheorem{notation}[theorem]{Notation}
\newtheorem{remark}[theorem]{Remark}

\newtheorem{example}[theorem]{Example}

\newcommand{\Lc}{\mathcal{L}}
\newcommand{\OO}{\mathcal{O}}

\newcommand{\C}{\mathscr{C}}

\def\Os{\mathscr{O}}

\newcommand{\PG}{\mathrm{PG}}

\newcommand{\EnG}{\mathrm{En}\Gamma}

\newcommand{\MM}{\mathbf{M}}
\newcommand{\Pf}{\mathbf{P}}

\newcommand{\F}{\mathbb{F}}

\newcommand{\A}{\mathfrak{A}}

\newcommand{\db}{\displaybreak[3]}
\def\t{\text}

\begin{document}
\title{Orbits of the class $\mathcal{O}_6$ of lines external with respect to the twisted cubic in $\mathrm{PG}(3,q)$
\thanks{The research of S. Marcugini and F. Pambianco was supported in part by the Italian
National Group for Algebraic and Geometric Structures and their Applications (GNSAGA -
INDAM) (Contract No. U-UFMBAZ-2019-000160, 11.02.2019) and by University of Perugia
(Project No. 98751: Strutture Geometriche, Combinatoria e loro Applicazioni, Base Research
Fund 2017-2019).}
}
\date{\vspace{-5ex}}
\maketitle
\begin{center}
{\sc Alexander A. Davydov}\\
 \emph{E-mail address:} alexander.davydov121@gmail.com\medskip\\
 {\sc Stefano Marcugini and
 Fernanda Pambianco }\\
 {\sc\small Department of  Mathematics  and Computer Science,  Perugia University,}\\
 {\sc\small Perugia, 06123, Italy}\\
 \emph{E-mail address:} \{stefano.marcugini, fernanda.pambianco\}@unipg.it
\end{center}

\begin{abstract}
In the projective space $\mathrm{PG}(3,q)$, we consider orbits of
lines under the stabilizer group of the twisted cubic. In the
literature, lines of $\mathrm{PG}(3,q)$ are
partitioned into classes, each of which is a union of line
orbits. We propose an
approach to obtain orbits of the class named $\mathcal{O}_6$, whose
complete classification is an
open problem. For all even and odd $q$ we describe  a
family of orbits of $\mathcal{O}_6$ and their stabilizer groups.
The orbits of this family include an essential part of all $\mathcal{O}_6$ orbits. 



 \end{abstract}

\section{Introduction}

In the three-dimensional projective space $\PG(3,q)$ over a Galois field $\F_{q}$ with $q$ elements, the normal rational curve $\mathscr{C}$,  named twisted cubic, has as many as $q+1$ points. Up to a change of the projective frame of $\PG(3,q)$, these points are  $P_t=(t^3,t^2,t,1)$, $t\in \F_{q}$, together with $P_\infty=(1,0,0,0)$. In particular, they form a complete $(q+1)$-arc in $\PG(3,q)$. This is a well known, relevant property which has already proved to be useful not only from a theoretical point of view but also for practical applications in cryptography; see for instance \cite{BonPolvTwCub,BrHirsTwCub,CasseGlynn82,CasseGlynn84,CKS,CosHirsStTwCub,GiulVincTwCub,KorchLanzSon,ZanZuan2010}.
A novel application of twisted cubic aimed at the construction of covering codes has been the motivation for the study of certain submatrices of the point-plane incidence matrix of $PG(3,q)$ arising from the action of the stabilizer group $G_q\cong \mathrm{PGL}(2,q)$ of $\mathscr{C}$ in $\PG(3,q)$.
The investigation, based on the known classification of the point and plane orbits of $G_q$ given in \cite{Hirs_PG3q}, was initiated by D. Bartoli and the present authors in 2020 \cite{BDMP-TwCub} and produced optimal multiple covering codes. The results in \cite{BDMP-TwCub} were also an important ingredient to classify the cosets of the $[ q+1, q-3,5]_q 3$ generalized doubly-extended Reed-Solomon code of codimension $4$ by means of their weight distributions \cite{DMP_RSCoset}.

For the study of the plane-line and the point-line incidence matrices, an explicit description of line orbits is useful. In \cite{Hirs_PG3q}, a partition of the lines in $\PG(3,q)$ into classes is given, each of which is a union of line orbits under $G_q$; see also Section 2. Apart from one class denoted by $\OO_6$, the number and the structure of the orbits forming those unions are independently considered by distinct methods in
 \cite{DMP_OrbLine, DMP_OrbLine2, DMP_OrbLine_sub}  (for all $q\ge2$),  \cite[Section 7]{BlokPelSzo} (for all $q\ge23$), and \cite{GulLav}  (for finite fields of characteristic $> 3$); see also the references therein. The results on the line orbits from \cite{BlokPelSzo,DMP_OrbLine,GulLav} are in accordance with each other. They (together with ones from \cite{BrHirsTwCub,Hirs_PG3q}) are collected in \cite[Section~2.2, Table 1]{DMP_PointLineInc}. In  \cite{DMP_PointLineInc}   the relations between       \cite{DMP_OrbLine, DMP_PlLineInc, DMP_PointLineInc} and  \cite{GulLav} are described.

The class $\OO_6$  contains lines external to the twisted cubic such that they are not chords of the cubic and do not lie in its osculating planes. The complete classification of the line orbits in $\OO_6$ constitutes an open problem.

Using the representation of the line orbits in \cite{DMP_OrbLine}, for all $q\ge2$ and apart from the lines in class $\OO_6$, the \emph{plane-line} incidence matrix of $\PG(3,q)$ is given in \cite{DMP_PlLineInc} and  the  \emph{point-line} incidence matrix of $\PG(3,q)$ is obtained in \cite{DMP_PointLineInc}.
For $\OO_6$, in \cite{DMP_PointLineInc,DMP_PlLineInc}, the corresponding average and cumulative values are calculated.

In \cite{GulLav}, apart from the lines in class $\OO_6$, for odd $q\not\equiv0\pmod3$ the numbers of distinct planes through distinct lines (called ``the plane orbit distribution of a line") and the numbers of distinct points lying on distinct lines
(called ``the point orbit distribution of a line") in of $\PG(3,q)$ are obtained.
For finite fields of characteristic $> 3$, the results of \cite{GulLav} on ``the plane orbit distribution of a line" and ``the point orbit distribution of a line'' are in accordance with those from \cite{DMP_PointLineInc,DMP_PlLineInc} on the plane-line and the point-line incidence matrices.
  
In  $\PG(3,q)$, for even $q=2^n$, $n\ge3$
the $(q+1)$-arc    $\mathcal{A}=\{(1,t,t^{2^h},t^{2^h+1})|t\in\F_q\} \cup
{(0,0,0,1)}$ with $gcd(n,h)=1$ (twisted cubic for $h=1$), has been considered  in a recent paper \cite{CePa},
where it is shown that the orbits of points and of planes
under the projective stabilizer $G_h$ of $\mathcal{A}$ are similar to those under $G_1$ described in \cite{Hirs_PG3q}; moreover, the point-plane incidence matrix with
respect to $G_h$-orbits  mirrors the case h=1 described in  \cite{BDMP-TwCub}. In \cite{CePa}, it is also proved that for even $q$,  $q\equiv\xi\pmod3$, $\xi \in \{1,-1\}$, $G_h$ has $2q+7+\xi$  orbits on lines, providing a proof of a conjecture of ours \cite[Conjecture 8.2]{DMP_OrbLine_sub,DMP_OrbLine} in the case even $q$. 

In this paper, we propose an approach to obtain orbits of the class $\OO_6$. Our approach is based on the analysis of the stabilizer of a line from $\OO_6$.
For all even and odd q we describe a family of orbits of
$\OO_6$ and their stabilizer groups. The orbits of this family include an essential part of all
$\OO_6$ orbits.


The paper is organized as follows.
Section \ref{sec:prelim} contains background and preliminaries. In particular, in Theorem \ref{th2:MAGMA}, the known computer results on line orbits of the class $\OO_6$ are given. \\
In Section \ref{sec:orbit}, the stabilizer group and the orbit of the line through the points $\Pf(1,0,0,1)$, $\Pf(0,0,1,0)$ is described for all the values of $q$ for which such a line belongs to $\OO_6$. \\
In Section \ref{sec:linesLmu}, the family of lines  $\{\ell_{\mu} ~|  ~\mu\in\F_q \}$ such that the points $\Pf(0,\mu,0,1), \Pf(1,0,1,0) \in \ell_\mu$ is defined, and for all $q$ it is proven when $\ell_{\mu} \in \OO_6$. The stabilizer groups and the orbits of the lines $\ell_\mu$ are analyzed in the following sections. \\
In Section \ref{sec:orbitlinesLmuEven}, for even $q$, the stabilizer groups, and the size and the number of orbits of  $\ell_{\mu}$-lines is established. Also it is stated that for even $q$ no $\ell_{\mu}$-line belongs to the orbit described in Section \ref{sec:orbit}. \\
In Section \ref{sec:orbitlinesLmuChar3}, for $q \equiv0\pmod3$, the stabilizer groups and the size of orbits of  $\ell_{\mu}$-lines is established. The exact number of orbits is found for $q\equiv-1\pmod4$; for $q\equiv1\pmod4$ a lower and an upper bound are given.\\
In Section \ref{sec:orbitlinesLmuOdd}, for $q$ odd, $q\not \equiv0\pmod3$, the stabilizer groups and the size of orbits of  $\ell_{\mu}$-lines is established. Also it is stated when an $\ell_{\mu}$-line belongs to the orbit described in Section \ref{sec:orbit}. \\

\section{Preliminaries}\label{sec:prelim}

In this section, in the beginning we cite some results from \cite{Hirs_PG3q,Hirs_PGFF} useful in this paper. Then, in Theorem \ref{th2:MAGMA}  we give computational results from \cite{DMP_OrbLine}.

Let $\boldsymbol{\pi}(c_0,c_1,c_2,c_3)$ be the plane of $\PG(3,q)$ with equation
$c_0x_0+c_1x_1+c_2x_2+c_3x_3=0$, $c_i\in\F_q$.
We denote $\F_{q}^*=\F_{q}\setminus\{0\}$, $\F_q^+=\F_q\cup\{\infty\}$.

Let $\Pf(x_0,x_1,x_2,x_3)$ be a point of $\PG(3,q)$  with homogeneous coordinates $x_i\in\F_{q}$.

Let  $P(t)$ be a point of $\in\PG(3,q)$ with
\begin{equation}\label{eq1:Pt}
  t\in\F_q^+,~ P(t)=\Pf(t^3,t^2,t,1) \t{ if }t\in\F_q,~P(\infty)=\Pf(1,0,0,0).
\end{equation}

Let $\C\subset\PG(3,q)$ be the \emph{twisted cubic} consisting of $q+1$ points no four of which are coplanar.
We consider $\C$ in the canonical form
\begin{equation}\label{eq2:cubic}
\C=\{P(t)\,|\,t\in\F_q^+\}.
\end{equation}

A \emph{chord} of $\C$ through the points $P(t_1)$ and $P(t_2)$ is a line joining either a pair of real points of $\C$ or a pair of complex conjugate points. In the last  case, it is an \emph{imaginary chord}. If the real points are distinct, it is a \emph{real chord}; if they coincide with each other, it is a \emph{tangent.} The coordinate vector of a chord has the form
\begin{equation}\label{eq1:chord}
 L_{\t{ch}}=(a_2^2, a_1 a_2, a_1^2-a_2, a_2, -a_1, 1)
\end{equation}
where $a_1 =t_ 1 + t_ 2$, $a_2 =t_1 t_ 2$. If $x^2- a_1x + a_2$ has 2, 1, or 0  roots in $\F_q$ then \eqref{eq1:chord} gives, respectively, a real chord, a tangent, or an imaginary chord.

The \emph{osculating plane} $\pi_\t{osc}(t)$ in the  point $P(t)\in\C$ has the form
\begin{equation}\label{eq2:osc_plane}
\pi_\t{osc}(t)=\boldsymbol{\pi}(1,-3t,3t^2,-t^3)\t{ if }t\in\F_q; ~\pi_\t{osc}(\infty)=\boldsymbol{\pi}(0,0,0,1).
\end{equation}
 The $q+1$ osculating planes form the osculating developable $\Gamma$ to $\C$, that is a \emph{pencil of planes} for $q\equiv0\pmod3$ or a \emph{cubic developable} for $q\not\equiv0\pmod3$.

A line is an \emph{axis}, intersection of a pair of real or complex conjugate planes of $\Gamma$, say $\pi_\t{osc}(t_1)$ and $\pi_\t{osc}(t_2)$, if its coordinate vector has the form
 \begin{equation}\label{eq1:axis}
   L_{\t{ax}}=(\beta_2^2, \beta_1 \beta_2, 3\beta_2, (\beta_1^2-\beta_2)/3, -\beta_1,1)
 \end{equation}
 where $\beta_1 = t_1+t_2$, $\beta_2 = t_1t_2$. We call the line a \emph{real axis}, a \emph{generator (tangent)} or
an \emph{imaginary axis} of $\Gamma$ as $x^2- \beta_1 x + \beta_2$ has 2, 1, or 0 roots in $\F_q$.

The null polarity $\A$ \cite[Sections 2.1.5, 5.3]{Hirs_PGFF}, \cite[Theorem~21.1.2]{Hirs_PG3q} is given by
\begin{equation}\label{eq2:null_pol}
\Pf(x_0,x_1,x_2,x_3)\A=\boldsymbol{\pi}(x_3,-3x_2,3x_1,-x_0),~q\not\equiv0\pmod3.
\end{equation}
It interchanges $\C$ and $\Gamma$, and their corresponding chords and axes.

\begin{notation}
Throughout the paper, we consider $q\equiv\xi\pmod3$ with $\xi\in\{-1,0,1\}$. Many values depend of $\xi$ or make sense only for specific $\xi$.
If it is not clear by the context, we note this by remarks.
The following notation is used.
\begin{align*}
  &G_q && \t{the group of projectivities in } \PG(3,q) \t{ fixing }\C;\db  \\
&tr&&\t{the sign of transposition};\db \\
&\#S&&\t{the cardinality of a set }S;\db\\
&\overline{AB}&&\t{the line through the points $A$ and }B;\db\\
&\triangleq&&\t{the sign ``equality by definition"};\db\\
&\EnG\t{-line}&&\t{a line, external respect to the cubic $\C$, not in osculating planes,}\db\\
&&&\t{and other than a chord and an axis;}\db \\
&\OO_6=\OO_{\EnG}&&\t{the union (class) of all orbits of $\EnG$-lines under }G_q.
\end{align*}
\end{notation}

\begin{remark}
  The words ``and an axis'' are included to the definition of $\EnG$-line by context of \cite[Lemma 21.1.4]{Hirs_PG3q}. 
\end{remark}
\begin{theorem}\label{th2_Hirs}
\emph{\cite[Chapter 21]{Hirs_PG3q}} The following properties of the twisted cubic $\C$ of \eqref{eq2:cubic} hold:
\begin{description}
  \item[(i)] The group $G_q$ acts triply transitively on $\C$;  $G_q\cong PGL(2,q)$ for $q\ge5$.
    A matrix $\MM$ corresponding to a projectivity of $G_q$ has the general form
  \begin{align}\label{eq2_M}
& \mathbf{M}=\left[
 \begin{array}{cccc}
 a^3&a^2c&ac^2&c^3\\
 3a^2b&a^2d+2abc&bc^2+2acd&3c^2d\\
 3ab^2&b^2c+2abd&ad^2+2bcd&3cd^2\\
 b^3&b^2d&bd^2&d^3
 \end{array}
  \right],a,b,c,d\in\F_q, ad-bc\ne0.
\end{align}

  \item[(ii)]  The lines of $\PG(3,q)$ can be partitioned into classes called $\OO_i$ and $\OO'_i=\OO_i\A$, each of which is a union of orbits under $G_q$, see \emph{\cite[Lemma 21.1.4]{Hirs_PG3q}} for details. In particular, all orbits of $\EnG$-lines form the class $\OO_6$ (we call it also $\OO_{\EnG}$) of the size $\#\OO_6=\#\OO_{\EnG}=(q^2-q)(q^2-1).$ If $q\not\equiv0\pmod3,~ q\ge5$, then $\OO'_6=\OO_6\A=\OO_6$.
\end{description}
\end{theorem}

\begin{lemma}\label{lem1:equation}
In $\F_q$, $q$ odd, the equation $3x^2+ 1= 0$  has:
 no roots, if $q\equiv-1\pmod3$, or 2 distinct roots, if $q\equiv 1\pmod3$.
\end{lemma}

\begin{proof}
By \cite[Section 1.5(xi)(xii)]{Hirs_PGFF}, $-3$ is a non-zero square (resp. a non-square) in $\F_q$ if $q\equiv1\pmod3$ (resp. $q\equiv-1\pmod3$).
\end{proof}

The following theorem states computational results regarding the orbits of $\EnG$-lines for $5\le q\le 37$ and $q=64$.

\begin{theorem} \label{th2:MAGMA}
\emph{\cite[Section 8]{DMP_OrbLine}}
Let $q\equiv\xi\pmod3$, $\xi\in\{1,-1,0\}$.
\begin{description}
  \item[(i)] Let $5\le q\le 37$ and $q=64$. Then

$\mathbf{(b)}$ The total number of line orbits in $\PG(3,q)$ is $2q+7+\xi$.

$\mathbf{(a)}$  For the total number $L_{\EnG\Sigma}$ of orbits of $\EnG$-lines we have\\
$L_{\EnG\Sigma}=2q-3+\xi\t{ for }q\t{ odd},~L_{\EnG\Sigma}=2q-2+\xi\t{ for }q\t{ even}$.

   \item[(ii)] Let $q$ be odd, $5\le q\le 37$.
   Then under $G_q$, for $\EnG$-lines,  there are\\
   $
   \begin{array}{lcl}
     (q-\xi)/3 & \t{orbits of length} & q^3-q, \\
     q-1 & \t{orbits of length} & (q^3-q)/2, \\
     n_q^{(\xi)} & \t{orbits of length} & (q^3-q)/4,
   \end{array}
   $\\
   where $n_q^{(1)}=(2q-11)/3,~n_q^{(-1)}=(2q-10)/3,~
    n_q^{(0)}=(2q-6)/3$.\\
     In addition, for $q\in\{7,13,19,25,31,37\}$ where $q\equiv1\pmod3$, there are\\
     $
     \begin{array}{lcl}
       1 &\t{orbit of length}&(q^3-q)/12, \\
       2&\t{orbits of length}&(q^3-q)/3.
     \end{array}
     $

  \item[(iii)] Let $q=8,16,32,64$. Then under $G_q$, for $\EnG$-lines, there are\\
     $2+\xi$  orbits of length $(q^3-q)/(2+\xi)$;
     $2q-4$ orbits of length $(q^3-q)/2$.
\end{description}
\end{theorem}

\begin{conjecture} \label{conj2:orbEnG} \emph{\cite{DMP_OrbLine}}
The results of Theorem \ref{th2:MAGMA} hold for all $q\ge5$ with the corresponding parity and $\xi$ value.
\end{conjecture}
For odd $q\not\equiv0\pmod3$, the conjecture on the case (i) of Theorem \ref{th2:MAGMA} is given also in \cite{GulLav}.

\section{An orbit $\Os_\Lc$ of the class $\OO_6=\OO_{\EnG}$ for $q\not\equiv0\pmod3$}\label{sec:orbit}
In this section $q\not\equiv0\pmod3$.

Let $Q_\beta$ and $Q_\infty$ be the points such that $Q_\beta=\Pf(1,0,\beta,1), ~\beta\in\F_q$; $Q_\infty=\Pf(0,0,1,0)$. We consider the line $\Lc = \overline{Q_0Q_\infty}$ through the points $Q_0$ and $Q_\infty$. We have
\begin{equation}\label{eq3:ell0infdef}
 \Lc =\overline{\Pf(1,0,0,1)\Pf(0,0,1,0)}= \{\Pf(0,0,1,0), \Pf(1,0,\beta,1)|\beta\in\F_q\}.
\end{equation}

\begin{lemma}\label{lem3:l0infInEnG}
The line  $\Lc$ is an $\EnG$-line.
\end{lemma}

\begin{proof}
For every $q$, the points $Q_\infty$ and $Q_\beta$, $\beta\in\F_q$, do not belong to the cubic $\C$, so  $\Lc$ is an external line.

  We show that no osculating plane $\pi_\t{osc}(t)$ of \eqref{eq2:osc_plane} contains two distinct points of $\Lc$. We have
$Q_\infty\in\pi_\t{osc}(\infty)$, $Q_\infty\in\pi_\t{osc}(0)$, $Q_\beta\notin\pi_\t{osc}(\infty)\cup\pi_\t{osc}(0),~\beta\in\F_q$.
Let $t\in\F_q^*$. Obviously, $Q_\infty\notin\pi_\t{osc}(t)$.
Points $Q_i,Q_j$ simultaneously belong to $\pi_\t{osc}(t)$ if and only if $1+3it^2-t^3=1+3jt^2-t^3=0$ that implies $it^2=jt^2$. As $t\ne0$ we have $i=j$.

Finally, the coordinate vector of $\Lc$ is $(0,-1,0,0,0,1)$.
Comparing it with \eqref{eq1:chord} and \eqref{eq1:axis}, we see that $\Lc$  cannot be an imaginary chord or an imaginary axis.
\end{proof}

Note that for $q\equiv0\pmod3$, $\Lc$ is not an $\EnG$-line. In fact, $\Lc$ is contained in the plane of equation $x_0-x_3=0$ that, if $q\equiv0\pmod3$, is the osculating plane $\pi_\t{osc}(1)$, see \eqref{eq2:osc_plane}.

We denote by $G_q^\infty$ the subgroup of  $G_q$  fixing the point $Q_\infty=\Pf(0,0,1,0)$. Let $\MM^\infty$ be a matrix corresponding to a projectivity of $G_q^\infty$.

\begin{lemma}\label{lem3:stabilQinf}
 The general form of the matrix $\MM^\infty$  is as follows:
 \begin{equation}\label{eq3:MM_Qinf}
   \MM^\infty=\left[
 \begin{array}{cccc}
 1&0&0&0\\
 0&d&0&0\\
 0&0&d^2&0\\
 0&0&0&d^3
 \end{array}
  \right],~d\in\F_q^*.
\end{equation}
\end{lemma}

 \begin{proof}
We find the version of the matrix $\MM$ of \eqref{eq2_M} fixing the point $\Pf(0,0,1,0)$. For $\delta\in\F_q^*$,  $\Pf(0,0,1,0)$ and $\Pf(0,0,\delta,0)$ represent the same point. We have
$[0,0,1,0]\MM=0,0,\delta,0],~\delta\in\F_q^*$,
that implies $3ab^2=0$, $b^2c+2abd=0,~ad^2+2bcd=\delta,~3cd^2=0$. If $a=b=0$ then $\delta = ad^2+2bcd=0$, contradiction. If $a=0,~b\ne0$ then $b^2c=0$ and $\delta = 2bcd=0$, contradiction. So, $a\ne0$, $b=0$. As $\MM$ is defined up to a factor of proportionality, we can put $a=1$. Now we have, $d^2=\delta\ne0,~3cd^2=0$, so $c=0$ and the assertion follows from \eqref{eq2_M} .
 \end{proof}

We want to determine the subgroup $G_q^{\Lc}$ of  $G_q$  fixing $\Lc$  and its orbit $\Os_\Lc$ under $G_q$. Let $\MM^{\Lc}$ be a matrix corresponding to a projectivity of $G_q^{\Lc}$.

\begin{lemma}\label{lem3:stabil0inf}
Let $q$ be even or let $-1/2$ be a non-cube in $\F_q$. Then  the general form of a matrix $\MM^{\Lc}$ corresponding to a projectivity of $G_q^{\Lc}$ is as follows:
 \begin{equation}\label{eq3:MM ell_0infty}
   \MM^{\Lc}=\left[
 \begin{array}{cccc}
 1&0&0&0\\
 0&d&0&0\\
 0&0&d^2&0\\
 0&0&0&d^3
 \end{array}
  \right],~d\in\F_q^*,~ d  \t{ is a cubic root of unity}.
\end{equation}
\end{lemma}

\begin{proof}
Let a projectivity $\psi\in G_q^{\Lc}$. We consider the case $Q_\infty \psi  =  Q_\beta$ for some $\beta\in\F_q$. The general form of a matrix $\MM$ corresponding to $\psi$ is given by \eqref{eq2_M}. We have $[0,0,1,0]\MM=[1,0,\beta,1]$
that implies $b^2c+2abd=0$, $ab^2=cd^2$, and $a,b,c,d \neq 0$. If $q$ is even, we have also $b^2c = 0$, contradiction.
Now consider the case $q$ odd. As $\MM$ is defined up to a factor of proportionality, we can put $b=1$.  From  $a=cd^2$ and $c+2ad=0$ we obtain $d^3=-1/2$, contradiction if $-1/2$ is not a cube in $\F_q$.

Thus, $Q_\infty \psi\ne Q_\beta$ with $\beta\in\F_q$. The only possible case is $Q_\infty \psi  =  Q_\infty$, see Lemma~\ref{lem3:stabilQinf}.
The matrix $\MM^{\Lc}$ must be the same form as $\MM^\infty$ \eqref{eq3:MM_Qinf} but the set of possible values of $d$ can be a proper subset of $\F_q^*$.
We should provide  $Q_0 \psi  =  Q_\beta$ for some $\beta\in\F_q$.
As $ [1,0,0,1]\MM^\infty = [1,0,0,d^3]$, it can happen only if $d^3=1$. 
\end{proof}

Remind that over $\F_q$, the equation $x^3=c$ has a unique solution, if $q\equiv  -1 \pmod3$, and 3 or $0$ solutions, if $q\equiv  1 \pmod3$, see\cite[Section 1.5(iv),(v)]{Hirs_PGFF}.

\begin{lemma}\label{lem3:stabil0infqm1}
\begin{description}
  \item[(i)] Let $q\equiv  -1 \pmod3$, $q$ odd. Then $G_q^{\Lc}$ has order $2$  and a matrix $\MM^{\Lc}$ corresponding to the non-trivial projectivity of $G_q^{\Lc}$ has the form \eqref{eq2_M} with
 \begin{equation*}
   a = \sqrt[3]{1/2},   ~b = 1,  ~ c = \sqrt[3]{2},   ~ d = -\sqrt[3]{1/2}.
 \end{equation*}
  \item[(ii)] Let $q\equiv  1 \pmod3$, $q$ odd and let  $-1/2$ be a cube in $\F_q$. Then $G_q^{\Lc}$ has order $12$  and is isomorphic to the group $A_4$. A matrix $\MM^\Lc$  of $G_q^\Lc$ has either the form \eqref{eq3:MM ell_0infty}
or the form \eqref{eq2_M} with
\begin{equation} \label{eq3:generalMatrix}
  a = \t{a cubic root of } 1/2,   ~ b = 1,  ~ c = -d/a^2,   ~ d = \t{a cubic root of } -1/2.
\end{equation}
\end{description}

\end{lemma}

\begin{proof}
\begin{description}
  \item[(i)] Let a projectivity $\psi\in G_q^{\Lc}$ and let  $\MM^\Lc$ be a matrix corresponding to $\psi$.
\begin{description}
  \item[(a)] Let $Q_\infty \psi  =  Q_\infty$. We have  $\MM^\Lc=\MM^\infty$, see Lemma \ref{lem3:stabilQinf} and \eqref{eq3:MM_Qinf}.  In~\eqref{eq3:MM_Qinf} we have $d=1$, so $\psi$ is the identity projectivity.

  \item[(b)] Let $Q_\infty \psi  =  Q_\beta$, $\beta\in\F_q$. The general form of  $\MM$ is given by \eqref{eq2_M}. We have
$[0,0,1,0]\MM=[1,0,\beta,1]$
that implies $b^2c+2abd=0$, $ab^2=cd^2$, and $a,b,c,d \neq 0$. As $\MM$ is defined up to a factor of proportionality, we can put $b=1$.  From $c+2ad=0$, $a=cd^2$,  we obtain $d^3=-1/2$.

Now for $b=1$ and $d^3=-1/2$, we consider  $Q_0 \psi$. The following holds
\begin{equation*}
[1,0,0,1]\MM = [a^3+1,~a^2c+d,~ac^2+d^2,~c^3-1/2].
\end{equation*}
\begin{description}
  \item[(b1)] Let $Q_0 \psi = Q_\infty$. Then $a^3+1=0$, $a^2c+d=0$, $ad-bc=ad-(-d/a^2)=0$, contradiction.
  \item[(b2)] Let   $Q_0 \psi = Q_\beta$, $\beta\in\F_q$. Then $a^3+1=c^3-1/2$ and $a^2c+d=0$ that implies $c = -d/a^2$, $2a^9+3a^6-1=0$. Putting $t=a^3$ we obtain $(t+1)^2(2t-1)=0$.
    \begin{description}
  \item[(b21)] Let $t+1=0$. Then $a^3= -1$  So, $c = ad$ and $ad-bc=0$, contradiction.
  \item[(b22)] Let $2t-1=0$. Then $t=1/2$ and
  $a=\sqrt[3]{1/2}$. So, $c = \sqrt[3]{2}$ and $ad-bc=-1- \sqrt[3]{2} \neq 0$.
\end{description}
\end{description}
\end{description}
  \item[(ii)] The assertions can be proved similarly to case (i) taking into account that if $-1/2$ is a cube then $1/2$ also is a cube. By direct computation using the computer algebra system Magma \cite{Magma}, a non trivial matrix of the form  \eqref{eq3:MM ell_0infty} has order three, whereas of the nine matrices of the form \eqref{eq3:generalMatrix}, three have order two and the other six have order three.
The only group of order $12$ having  three elements of order two and eight elements of order three is $A_4$, see \cite{groupbook}. \qedhere
 \end{description}
\end{proof}

\begin{theorem}\label{th3:main}
Let the $\EnG$-line $\Lc$  be as in \eqref{eq3:ell0infdef}, $q\equiv \xi\pmod3$. Let $G_q^{\Lc}$ be the subgroup of $G_q$ fixing  $\Lc$ and let $\Os_\Lc$ be the orbit of $\Lc$ under $G_q$. Then the sizes of $G_q^{\Lc}$ and $\Os_\Lc$ are as follows:
\begin{align*}\label{eq3:main}
  &(i)~ \xi=1,~q\text{ is even or }-\frac{1}{2}\text{ is a non-cube in }\F_q.
  && \#G_q^{\Lc}=3, ~ \#\Os_\Lc=\frac{1}{3}(q^3-q).\db\\
  &(ii)~ \xi=1,~q\text{ is odd and }-\frac{1}{2}\text{ is a cube in }\F_q. && \#G_q^{\Lc}=12,~\#\Os_\Lc=\frac{1}{12}(q^3-q),\db\\
  &~ \; && \#G_q^{\Lc} \cong A_4.\db\\
  &(iii)~ \xi=-1,~q\text{ is even}. && \#G_q^{\Lc}=1,~ \#\Os_\Lc=q^3-q.\db\\
  &(iv)~ \xi=-1,~q\text{ is odd}.  && \#G_q^{\Lc}=2 ,~ \#\Os_\Lc=\frac{1}{2}(q^3-q).
  \end{align*}
\end{theorem}

\begin{proof}
The sizes of  $G_q^{\Lc}$ follow from Lemmas \ref{lem3:stabil0inf}, \ref{lem3:stabil0infqm1} and the results of \cite[Section 1.5 (ii),(iii)]{Hirs_PGFF}. By \cite[Lemma 2.44(ii)]{Hirs_PGFF}, the size of  $\#\Os_\Lc$ is $\#G_q/\#G_q^{\Lc}$.
\end{proof}



\section{A family of $\EnG$-lines $\ell_\mu$, $\mu\in\F_q^{*} \setminus\{1, 1/9\}$}\label{sec:linesLmu}

Let $\mu\in\F_q$. Let $R_{\mu,\gamma}$  be the point such that
\begin{equation}\label{eq2:R}
  R_{\mu,\gamma}=\Pf(\gamma,\mu,\gamma,1),~\gamma\in\F_q^+;~R_{\mu,\infty}=\Pf(1,0,1,0).
\end{equation}
We consider the line $\ell_{\mu} = \overline{R_{\mu,0}R_{\mu,\infty}}$ through $R_{\mu,0}$ and $R_{\mu,\infty}$.
\begin{equation}\label{eq2:ellmu}
\ell_{\mu} = \overline{\Pf(0,\mu,0,1)\Pf(1,0,1,0)}= \{\Pf(\gamma,\mu,\gamma,1)|\gamma\in\F_q^+,~\mu\t{ is fixed},~\mu\in\F_q\}.
\end{equation}
The points of the line $\ell_{\mu}$ satisfy the equations
\begin{equation}\label{eq2:ellmueq}
x_0=x_2,~ x_1=\mu x_3,~\mu\t{ is fixed},~\mu\in\F_q.
\end{equation}
By  \eqref{eq2:ellmu}, the coordinate vector  $L_\mu$ of $\ell_{\mu}$  is
\begin{equation}\label{eq2:coordVectElMu}
  L_\mu = (\mu,0,1,-\mu,0,1).
\end{equation}

\begin{lemma}\label{lem2:L0L1}
Let $q\ge5$. The line $\ell_0$ is an unisecant of the cubic $\C$ not in an osculating plane. The line $\ell_1$ is a tangent if $q$ is even, a real chord if $q$ is odd.
\end{lemma}

\begin{proof}
The assertions follow from \eqref{eq1:Pt}, \eqref{eq2:osc_plane}, \eqref{eq2:ellmu}.
\end{proof}


\begin{lemma}\label{lem2:EllMu}
Let  $\mu\in\F_q^*\setminus\{1\}$. For all $q\ge5$, the line $\ell_\mu$ \eqref{eq2:ellmu} has the following properties:
\begin{description}
  \item[(i)]
The line $\ell_{\mu}$ is an external line to the twisted cubic $\C$.

  \item[(ii)]
The line $\ell_{\mu}$ is not a chord of the twisted cubic $\C$.

  \item[(iii)] The line $\ell_{\mu}$ does not lie in the osculating planes $\pi_\t{osc}(\infty)$, $\pi_\t{osc}(0)$. Moreover,
\begin{equation}\label{eq2:R_in notin}
   R_{\mu,\gamma}\notin\pi_\t{osc}(\infty),~\gamma\in\F_q;~ R_{\mu,\infty}\notin\pi_\t{osc}(0),~R_{\mu,\infty}\in\pi_\t{osc}(\infty).
\end{equation}

  \item[(iv)]
 Let $R_{\mu,\infty}\notin\pi_\t{osc}(t)$, $\forall t\in\F_q$. Then the line $\ell_{\mu}$ is not contained in any osculating plane.

  \item[(v)]  We have $R_{\mu,\infty}\in\pi_\t{osc}(t)$, $t\in\F_q$, if and only if $3t^2+1=0$.
\end{description}
\end{lemma}

\begin{proof}
\begin{description}
  \item[(i)]
By \eqref{eq1:Pt} and \eqref{eq2:ellmu}, no point of $\ell_{\mu}$ belongs to $\C$.

  \item[(ii)] Comparing the coordinate vectors \eqref{eq1:chord} and \eqref{eq2:coordVectElMu} one sees $a_2^2=\mu$ and $a_2=-\mu$ that implies $\mu=\mu^2$, contradiction as $\mu\in\F_q^*\setminus\{1\}$.

  \item[(iii),(v)]
The assertions follow from \eqref{eq2:osc_plane}, \eqref{eq2:R}.

  \item[(iv)]
 Together with \eqref{eq2:R_in notin}, the hypothesis means that for any osculating plane there is a point of  $\ell_\mu$  not belonging to the plane. The assertion follows. \qedhere
\end{description}
\end{proof}

\begin{lemma}\label{lem4:axis}
  Let $q\not\equiv0\pmod3$. Then for the line $\ell_\mu$ \eqref{eq2:ellmu}
the following holds:
  \begin{description}
    \item[(i)] A line $\ell_\mu$ is an imaginary axis if and only if  $q$ is odd, $q\equiv-1\pmod3$, $\mu={1/9}$;
    \item[(ii)] A line $\ell_\mu$ is a real axis if and only  if $q$ is odd, $q\equiv 1\pmod3$, $\mu={1/9}$.
\end{description}
\end{lemma}

\begin{proof}
 Considering   the vectors \eqref{eq1:axis} and \eqref{eq2:coordVectElMu}, we obtain $\beta_2^2 =\mu$, $\beta_2 = 1/3$, $\beta_1=0$. This implies $\mu = 1/9$. By Lemma \ref{lem1:equation}, the equation $x^2-\beta_1 x + \beta_2= x^2+ 1/3= 0$ has 2 distinct roots if $q$ is odd, $q\equiv 1\pmod3$, or 0 roots if $q$ is odd, $q\equiv-1\pmod3$. The assertions follow.
\end{proof}

\begin{lemma}\label{lem4:EnGEllMu}
  Let  $\mu\in\F_q^*\setminus\{1\}$. Let $q\ge5$. The line $\ell_{\mu}$ \eqref{eq2:ellmu} is an $\EnG$-line in the following cases:
  \begin{description}
  \item[(i)] $q$ is even;
  \item[(ii)] $q\equiv0\pmod3$;
 \item[(iii)]  $q\not\equiv0\pmod3$, $q$ is odd,  $\mu\ne1/9$.
\end{description}
\end{lemma}

  \begin{proof}
  We prove some properties of $\ell_{\mu}$ that are not included in Lemma \ref{lem2:EllMu}.
  By Lemma~\ref{lem2:EllMu}, it is sufficient to consider only the case $3t^2+1=0$.
 \begin{description}

   \item[(i)]  For even $q$, the equality $3t^2+1=0$ gives $t=1$.  Assume that $R_{\mu,\gamma}\in\pi_\t{osc}(1)$. Then, for even $q$,  by \eqref{eq2:osc_plane}, \eqref{eq2:R},   $1+\gamma +\mu+\gamma=0$ that implies  $\mu=1$, contradiction.
   \item[(ii)] The equality $3t^2+1=0$ comes to $0=1$, contradiction. The axis of $\Gamma$ has equations $x_0=0, x_3=0$, so by \eqref{eq2:ellmueq} it is not a line of type $\ell_{\mu}$ and it intersects no $\ell_{\mu}$ line.

   \item[(iii)]
   We apply Lemma \ref{lem1:equation}.
If $q\equiv-1\pmod3$, then $3t^2+1\ne0$, $\forall t\in\F_q$. However,  by Lemma \ref{lem4:axis}, $\ell_{1/9}$ is an imaginary axis.

   If $q\equiv1\pmod3$, let $3t^2+1=0$. Then $t^2=-1/3\ne0$. Assume   $R_{\mu,\gamma}\in\pi_\t{osc}(t)$,  $t\in\F_q^*$, see Lemma \ref{lem2:EllMu}(iii). Then, by \eqref{eq2:osc_plane}, \eqref{eq2:R},  $t^3-3\gamma t^2+3\mu t-\gamma=0$ that implies $t^3+3\mu t=0$, $\mu=1/9$, contradiction. \qedhere
 \end{description}
\end{proof}

Let $G_q^{\ell_\mu}$ be the subgroup of $G_q$ fixing  $\ell_\mu$. For $q$ odd,  $G_q^{\ell_{\mu}}$ has order at least two.

\begin{lemma}\label{lemS5:P0notfixedodd}
Let $q \ge5$ be odd. Let a projectivity $\psi\in G_q^{\ell_{\mu}}$ fix $R_{\mu,\infty}$. Then $\psi$  has a matrix of the form
 \begin{equation}\label{eq:M_Rinfodd}
   \MM(\psi)=
   \left[
 \begin{array}{cccc}
 1&0&0&0\\
 0&d&0&0\\
 0&0&d^2&0\\
 0&0&0&d^3
 \end{array}
  \right],~d\in\{1,-1\}.
\end{equation}
Moreover, $(\psi)^2=\psi$ and $R_{\mu,\gamma}\psi=R_{\mu,d\gamma},~\gamma\in\F_q$.
\end{lemma}

\begin{proof}
The matrix $\MM(\psi)$ is a version of $\MM$ of \eqref{eq2_M}.
By hypothesis,  $R_{\mu,\infty} \psi = R_{\mu,\infty} $ and $R_{\mu,0} \psi = R_{\mu,\gamma}$, $\gamma\in\F_q$, i.e. $[1,0,1,0]\MM = [1,0,1,0]$ and $[0,\mu,0,1]\MM=[\gamma,\mu,\gamma,1]$, $\gamma\in\F_q$,
that implies
\begin{align}\label{lemS5:P0notfixedeq1odd}
&a^2c+b^2c+2abd=0,~c^3+3cd^2=0,~ a^3+3ab^2 = ac^2+ad^2+2bcd;\db\\
&\label{lemS5:P0notfixedeq2odd}
3\mu a^2b+b^3=\mu (bc^2+2acd)+bd^2,~
\mu (a^2d+2abc)+b^2d = 3\mu^2 c^2d+\mu d^3.
\end{align}
\begin{description}
  \item[(a)] Let $c = 0$. As  $ad-bc=ad\ne0$, $a,d \neq 0$. From the first relation of \eqref{lemS5:P0notfixedeq1odd} we obtain $b=0$. Then the last relation of \eqref{lemS5:P0notfixedeq1odd} becomes $a(a-d)(a+d)=0$. As $\MM$ is defined up to proportionality, we can fix $a=1$, then $d = \pm 1$. For $c=b=0$, $a=1$, $d = \pm 1$, \eqref{lemS5:P0notfixedeq2odd} is also satisfied and  \eqref{eq2_M} becomes \eqref{eq:M_Rinfodd}.

  \item[(b)]  Let $c \neq 0$. The second relation of \eqref{lemS5:P0notfixedeq1odd} becomes $c^2+3d^2=0$.
      \begin{description}
        \item[(b1)] Let $q\equiv 0 \pmod 3$. Then $c^2=0$, contradiction.
        \item[(b2)]  Let $q\not \equiv 0 \pmod 3$.
      Then $d \neq 0$, otherwise again $c^2=0$, contradiction. As $\MM$ is defined up to proportionality, we can fix $d=1$. Then $q \equiv 1 \pmod 3$, as $c = \pm \sqrt{-3}$.
      From the first relation of \eqref{lemS5:P0notfixedeq1odd} we obtain
 $a = -3b/c$ or $a = b/c$.
 \begin{description}
   \item[(b21)] Let $a = -3b/c$. Then $ad-bc =0$, contradiction.
   \item[(b22)] Let $a = b/c$.  The third relation of \eqref{lemS5:P0notfixedeq1odd}  becomes
  $8b(b^2+3)/c^3=0$ that implies $b=0$ or $b^2 = -3$, $q \equiv 1 \pmod 3$, $b = \pm \sqrt{-3}$.

  If $b=0$ also $a=0$ and $ad-bc=0$, contradiction.

  If $b^2 = -3$, the  first relation of \eqref{lemS5:P0notfixedeq2odd}  implies $\mu = 1$, contradiction.
 \end{description}
      \end{description}
\end{description}
The last assertion follows  by direct computation.
\end{proof}

\section{The $q-2$ distinct   orbits of $\ell_\mu$-lines, $\mu\in\F_q^{*} \setminus\{1\}$, for even $q$}\label{sec:orbitlinesLmuEven}
In this section we consider the orbits of  the lines $\ell_\mu$ of type \eqref{eq2:ellmu}, $\mu\in\F_q^{*} \setminus\{1\}$,  for even  $q \geq 8$. In this case the matrix of \eqref{eq2_M}  corresponding to a projectivity of $G_q$  reduces to
  \begin{align}
 & \label{eq2:Meven} \mathbf{M}=\left[
 \begin{array}{cccc}
 a^3&a^2c&ac^2&c^3\\
 a^2b&a^2d&bc^2&c^2d\\
ab^2&b^2c&ad^2&cd^2\\
 b^3&b^2d&bd^2&d^3
 \end{array}
  \right],~a,b,c,d\in\F_q,~ ad-bc\ne0.
\end{align}

\begin{lemma}\label{lemS5:P0notfixed}
Let $q \ge 8$, $q$ even. Let $\psi\in G_q^{\ell_{\mu}}$ fix $R_{\mu,\infty}$. Then  $\psi = I$, the identity element of $G_q^{\ell_{\mu}}$.
\end{lemma}

\begin{proof}
A matrix corresponding to $\psi$ is a version of $\MM$ of \eqref{eq2:Meven}.
By hypothesis,  $R_{\mu,\infty} \psi = R_{\mu,\infty} $ and $R_{\mu,0} \psi = R_{\mu,\gamma}$, $\gamma\in\F_q$, i.e. $[1,0,1,0]\MM = [1,0,1,0]$ and $[0,\mu,0,1]\MM=[\gamma,\mu,\gamma,1]$,
that implies
\begin{align}\label{lemS5:P0notfixedeq1}
 &a^2c+b^2c=0,~c^3+cd^2=0,~ a^3+ab^2 = ac^2+ad^2;\db\\
 &\mu a^2b+b^3=\mu bc^2+bd^2,~ \mu a^2d+b^2d = \mu,~ \mu c^2d+d^3\ne0.\label{lemS5:P0notfixedeq2}
 \end{align}
\begin{description}
  \item[(a)] Let $c \neq 0$.
  \begin{description}
    \item[(a1)] Let $a = 0$. By \eqref{lemS5:P0notfixedeq1}, $b^2=0$ that implies  $ad-bc= 0$, contradiction.
    \item[(a2)]  Let $a \neq 0$.  By \eqref{lemS5:P0notfixedeq1},
 $(a+b)^2=0,~ (c+d)^2=0$, that implies,  $a=b$, $c=d$, and $ad-bc=0$, contradiction.
  \end{description}

  \item[(b)]  Let $c = 0$. As  $ad-bc=ad\ne0$, $a,d \neq 0$ and \eqref{lemS5:P0notfixedeq1} and \eqref{lemS5:P0notfixedeq2} become, respectively,
 \begin{align}\label{lemS5:P0notfixedeq3}
&a^2+b^2 = d^2;\db\\
&\label{lemS5:P0notfixedeq4}
\mu ba^2+b^3=bd^2, \mu a^2d+b^2d= \mu.
\end{align}
\begin{description}
  \item[(b1)] Let $b=0$. By \eqref{lemS5:P0notfixedeq3}, $a^2+ d^2=(a+d)^2=0$, i.e.  $a=d$. Then from \eqref{lemS5:P0notfixedeq4} $\mu d^3= \mu$, so $d^3= 1$ i.e. d is a cubic root of 1 and by \eqref{eq2:Meven} $\MM$ is the identity matrix. Therefore $\psi = I$, the identity element of $G_q^{\ell_{\mu}}$.
  \item[(b2)] Let $b\neq0$. From \eqref{lemS5:P0notfixedeq4} we obtain $\mu a^2 +b^2=d^2$, and from \eqref{lemS5:P0notfixedeq3} $\mu a^2+a^2+ d^2=d^2$, so $\mu a^2+a^2=0$  that implies  $\mu = 1$, contradiction. \qedhere
\end{description}
\end{description}
\end{proof}

\begin{theorem}\label{S5:stab}
Let $q$ be even. Then the size of the subgroup $G_q^{\ell_{\mu}}$  of $G_q$ fixing the $\EnG$-line $\ell_{\mu}$ is $\#G_q^{\ell_{\mu}}=2$. The size of the orbit $\Os_\mu$ of $\ell_{\mu}$ under $G_q$ is equal to $\#\Os_\mu=(q^3-q)/2.$
The non-trivial element $\psi$ of $G_q^{\ell_{\mu}}$ has a matrix of the form:
 \begin{equation}\label{S5:stab_matqeven}
   \MM^{\ell_{\mu}}=\left[
 \begin{array}{cccc}
 0&0&0&1\\
 0&0&\sqrt{\mu}&0\\
 0&\mu&0&0\\
 \sqrt{\mu^3}&0&0&0
 \end{array}
  \right].
\end{equation}
Moreover  $R_{\mu,\infty} \psi = R_{\mu,0}$ and vice versa $ R_{\mu,0}\psi =R_{\mu,\infty}$.
\end{theorem}

\begin{proof}
Let  $\psi\in G_q^{\ell_{\mu}}$, $\psi \neq I$,  the identity element of $G_q^{\ell_{\mu}}$ and let  $\MM^{\ell_{\mu}}$ be a matrix corresponding to $\psi$. 
We find it as is a version of $\MM$ of \eqref{eq2:Meven}.
By Lemma \ref{lemS5:P0notfixed},   $R_{\mu,\infty} \psi = R_{\mu,\gamma}$, $\gamma \in \F_q$,  i.e. $[1,0,1,0]\MM=[\gamma,\mu,\gamma,1]$
that implies
\begin{equation}\label{ThS5:stab_eq2}
a^3+ab^2=ac^2+ad^2, \;\;  a^2c+b^2c = \mu, \;\;  c^3+cd^2=1.
 \end{equation}
The last relation of \eqref{ThS5:stab_eq2} gives $c \neq 0$. As $\MM$ is defined up to proportionality, we can fix $c=1$.
\begin{description}
  \item[(a)]Let $a=0$. By \eqref{ThS5:stab_eq2},  $b^2 =\mu$,  $d = 0$. As $q$ is even, every element of $\F_q$ is a square and has exactly one square root, so $\MM$ has the form \eqref{S5:stab_matqeven} and  $R_{\mu,\infty} \psi = R_{\mu,0}$. By direct computation, $\MM$ has order 2. This implies $R_{\mu,0} \psi = R_{\mu,\infty} (\psi)^2  =R_{\mu,0}$, so $\ell_{\mu} \psi = \ell_{\mu}$. For the size of the orbit we apply \cite[Lemma 2.44(ii)]{Hirs_PGFF}.
  \item[(b)] Let $a \neq 0$. Then  \eqref{ThS5:stab_eq2} becomes
$a^2+b^2=1+d^2=1,~ a^2+b^2 = \mu $ that implies $\mu=1$, contradiction. \qedhere
\end{description}
\end{proof}

\begin{theorem}\label{S5:th_numberorbits}
Let $q$ be even. Then any two lines $\ell_{\mu'}$, $\ell_{\mu''}$ of type \eqref{eq2:ellmu}, $\mu', \mu'' \in \F_q^{*} \setminus\{1\}, \mu' \neq \mu''$, belong to different orbits of $G_q.$ No  $\ell_{\mu}$-line of type \eqref{eq2:ellmu} belongs to the orbit $\Os_\Lc$ of the line $\Lc$.
\end{theorem}
\begin{proof}
Let $\ell_{\mu'}\psi\ = \ell_{\mu''}$, $\psi\in G_q$. Then by Lemma \ref{lemS5:P0notfixed}, $R_{\mu',\infty} \psi \neq R_{\mu',\infty}=R_{\mu'',\infty}$. Reasoning as in the proof of Theorem \ref{S5:stab}, we have that  $R_{\mu',\infty} \psi = R_{\mu'',0}$ and a matrix  corresponding to $\psi$ is
 \begin{equation*}
   \MM(\psi)=\left[
 \begin{array}{cccc}
 0&0&0&1\\
 0&0&\sqrt{\mu''}&0\\
 0&\mu''&0&0\\
 \sqrt{\mu''^3}&0&0&0
 \end{array}
  \right].
\end{equation*}

Then $[0, \mu', 0, 1]\MM(\psi)=[\sqrt{\mu''^3},0,\mu'\sqrt{\mu''},0]$ that is equal to $R_{\mu'',\infty}$. This implies $\sqrt{\mu''^3} = \mu'\sqrt{\mu''}$, so $\mu' = \sqrt{\mu''^2}$. As $q$ is even, by \cite[Lemma 2.44(viii)]{Hirs_PGFF}. $\mu' = \mu''$, contradiction. Finally, by Theorems \ref{th3:main} and \ref{S5:stab}, the lines $\Lc$ and  $\ell_{\mu}$ have different stabilizer groups.
\end{proof}

\begin{remark}
  By Theorems \ref{th3:main}(i)(iii), \ref{S5:stab}, and \ref{S5:th_numberorbits}, for even $q$, we obtained sizes and structures of $q-1$ orbits,
  which contain in total $(q-2)(q^3-q)/2+\#\Os_{\Lc}$ $\EnG$-lines where $\#\Os_{\Lc}=(q^3-q)/3$ if $q\equiv1\pmod3$ and $\#\Os_{\Lc}=q^3-q$ if $q\equiv-1\pmod3$. This is  approximately   half of all $\EnG$-lines, see $\#\OO_6=\#\OO_{\EnG}=(q^2-q)(q^2-1)$ in Theorem~\ref{th2_Hirs}(ii).
\end{remark}

\section{On the orbits of lines $\ell_\mu$, $\mu\in\F_q^*\setminus\{1\}$, $q\equiv0\pmod3$} \label{sec:orbitlinesLmuChar3}
For $q\equiv0\pmod3$, a matrix $\MM$ corresponding to a projectivity of $G_q$ has the general form
\begin{align}\label{eq:M_q=0}
\mathbf{M}=\left[
 \begin{array}{cccc}
 a^3&a^2c&ac^2&c^3\\
 0&a^2d-abc&bc^2-acd&0\\
 0&b^2c-abd&ad^2-bcd&0\\
 b^3&b^2d&bd^2&d^3
 \end{array}
  \right],~a,b,c,d\in\F_q,~ ad-bc\ne0.
\end{align}

By \eqref{eq:M_q=0}, \eqref{eq2:R},
  we have
\begin{align}\label{eq:infM}
&R_{\mu,\infty}\MM=[1,0,1,0]\MM=[a^3,a^2c+b^2c-abd,ac^2+ad^2-bcd,c^3];\db\\\
\label{eq:0xM}
&R_{\mu,0}\MM=[0,\mu,0,1]\MM=[b^3,\mu(a^2d-abc)+b^2d,\mu(bc^2-acd)+bd^2,d^3].
\end{align}

\begin{lemma}\label{lem:infTOinf_q=0}
Let $q\equiv0\pmod3$.   Let  $\psi\in G_q$  fix $ R_{\mu,\infty}$. 
Then a matrix $\MM(\psi)$ corresponding to $\psi$ has the general form
 \begin{equation}\label{eq:M_Rinf_q=0}
   \MM(\psi)=
   \left[
 \begin{array}{cccc}
 1&0&0&0\\
 0&d&0&0\\
 0&0&d^2&0\\
 0&0&0&d^3
 \end{array}
  \right],~d\in\{1,-1\}.
\end{equation}
Moreover, for any point $P\in\PG(3,q)$, we have $P\psi=P$ if $d=1$ and $P(\psi)^2=P$ if $d=-1$
. Also,  $R_{\mu,\gamma}\psi=R_{\mu,d\gamma},~\gamma\in\F_q$.
\end{lemma}

\begin{proof}
Let $\MM$  be given by \eqref{eq:M_q=0}. By hypothesis, $[1,0,1,0]\MM=[1,0,1,0]$. Comparing components of the  vectors $[1,0,1,0]\MM$ and $[1,0,1,0]$ we consequently obtain
 $c=0$, $a\ne0$, $abd=0$, $ad^2\ne0$, $d\ne0$, $b=0$, $a^3=ad^2$,
 $d^2=a^2$. As $\MM$ is defined by the factor of proportionality, we can put $a=1$. Then $d=\sqrt{1}\in\{1,-1\}$, and from $\MM$ we obtain $\MM(\psi)$.

The remaining assertions are obvious.
\end{proof}

\begin{lemma}\label{lem:infTOgam_q=0}
Let $q\equiv0\pmod3$. Let $\varphi \in G_q^{\ell_\mu}$ be such that $ R_{\mu,\infty}\varphi=R_{\mu,\gamma}$, $\gamma\in\F_q$. If $\mu$ is not a square in $\F_q$ then $\varphi$ does not exist. If $\mu$ is a square in $\F_q$ then a matrix $\MM(\varphi)$ corresponding to $\varphi$ has the general form
 \begin{equation}\label{eq:M_Rinf gam_q=0}
   \MM(\varphi)=
   \left[
 \begin{array}{cccc}
 0&0&0&1\\
 0&0&b&0\\
 0&b^2&0&0\\
 b^3&0&0&0
 \end{array}
  \right],~b\in\{\sqrt{\mu},-\sqrt{\mu}\}.
\end{equation}
Moreover, for any point $P\in\PG(3,q)$, we have $P(\varphi)^2=P$. Also,
\begin{equation*}
R_{\mu,\infty}\MM(\varphi)=R_{\mu,0},~R_{\mu,0}\MM(\varphi)=R_{\mu,\infty},~R_{\mu,\gamma}\MM(\varphi)=R_{\mu,\gamma'},~\gamma,\gamma'\in\F_q^*,~\gamma'=
\frac{\pm\sqrt{\mu^3}}{\gamma}.
\end{equation*}
\end{lemma}

\begin{proof}
Let $\MM$  be given by \eqref{eq2_M}. By  hypothesis, there exist $\gamma,\zeta\in\F_q$ such that $[1,0,1,0]\MM=[\gamma,\mu,\gamma,1]$ and $[\zeta, \mu ,\zeta ,1]\MM=[1,0,1,0]$. The first equality implies $c^3\neq 0$, see \eqref{eq:infM}. As $\MM$ is defined up to a factor of proportionality, we can put $c=1$. Now,
comparing components of the  vectors $[1,0,1,0]\MM$ with $[1,\gamma,1,\gamma]$ and $[\zeta, \mu ,\zeta ,1]\MM$ with $[1,0,1,0]$, we obtain
\begin{align}\label{eq:Im_mu_infty_1}
&a^3 = a + ad^2 - bd;\db\\
&\label{eq:Im_mu_infty_2}
\mu = a^2+ b^2 -abd.\db\\
&\label{eq:Im_mu_infty_3}
\zeta a^3+ b^3 =  -\mu ad + \mu b + \zeta a + \zeta ad^2 -\zeta bd + bd^2\db\\
&\label{eq:Im_mu_infty_4}
\mu a^2d -\mu ab + \zeta a^2 -\zeta abd + \zeta b^2 + b^2d = 0;\db\\
&\label{eq:Im_mu_infty_5}
 \zeta +  d^3=0.
\end{align}
\begin{description}
  \item[(a)]
 Let $\zeta = 0$. Then, by \eqref{eq:Im_mu_infty_5},  $d=0$. By  \eqref{eq:Im_mu_infty_1}, $a(a-1)(a+1)=0.$
 \begin{description}
   \item[(a1)] Let $a=0$. By \eqref{eq:Im_mu_infty_2}, $\mu$ is a square and $b = \pm \sqrt{\mu}$. In this case, i.e. for $\zeta =a=0,$ $b = \pm \sqrt{\mu}$, also \eqref{eq:Im_mu_infty_3}, \eqref{eq:Im_mu_infty_4}, are satisfied.
   \item[(a2)] Let $a = \pm 1$. Then  \eqref{eq:Im_mu_infty_4}
becomes $\mp \mu b= 0$ that implies $b= 0$ and $ad-bc=0$, contradiction.
 \end{description}

  \item[(b) ]Let  $\zeta \neq 0$. By \eqref{eq:Im_mu_infty_5},  $\zeta=-d^3$, $d=\sqrt[3]{-\zeta}\neq0$. From \eqref{eq:Im_mu_infty_1}, $b=(a+ad^2-a^3)/d$ that implies
$ad-bc=  (a^3-a)/d.$  If $a=0, \pm1$  then $ad-bc=0$, contradiction.
 For the obtained $b$, by \eqref{eq:Im_mu_infty_2}, $\mu=a^2((a^2-d^2)(a^2+1)+1)/d^2$. If $a=\pm d$ then $\mu =  1$, contradiction. For the obtained $\zeta$, $b$, and $\mu$, by \eqref{eq:Im_mu_infty_3} and \eqref{eq:Im_mu_infty_4}, we have
 \begin{align*}
& -a^2(a+1)(a-1)(a+d)(a-d)(a^2+d^2+1)/d=0;\notag\db\\
&a^2(a+1)(a-1)(a+d)(a-d)(a^4 + a^2 - d^4 + d^2 + 1)/d^3=0.
\end{align*}
By above, $a \neq 0,\pm 1, \pm d$. Therefore $a^2+d^2+1=0$ and
$a^4 + a^2 - d^4 + d^2 + 1 = 0$ that implies $a^4 - d^4 = (a+d)(a-d)(a^2+d^2)=0$, $a^2+d^2=0$, contradiction.

The remaining assertions can be obtained by direct computation. \qedhere
\end{description}
\end{proof}

\begin{theorem}\label{th:stab_l_mu_q=0}
Let $q\equiv0\pmod3$.
\begin{description}
  \item[(i) ] Let $\mu$ be a non-square in $\F_q$. Then the size of the subgroup $G_q^{\ell_\mu}$
 of $G_q$ fixing the $\EnG$-line $\ell_\mu$ is $\#G_q^{\ell_\mu}=2$. The length of the orbit of $\ell_\mu$ under $G_q$ is equal to $(q^3-q)/2$. The
elements of $G_q^{\ell_\mu}$ have the matrix form \eqref{eq:M_Rinf_q=0}.

  \item[(ii)] Let $\mu$ be a square in $\F_q$. Then  the subgroup $G_q^{\ell_\mu}$ of $G_q$ fixing the $\EnG$-line $\ell_\mu$ is isomorphic to $C_2 \times C_2$. The length of the orbit of $\ell_\mu$ under $G_q$ is equal to $(q^3-q)/4$. Two elements of $G_q^{\ell_\mu}$ have the matrix form \eqref{eq:M_Rinf_q=0} and the other two ones are given by \eqref{eq:M_Rinf gam_q=0}.
\end{description}
\end{theorem}

\begin{proof}
  The assertions follow from Lemmas \ref{lem:infTOinf_q=0}, \ref{lem:infTOgam_q=0} and \cite[Lemma 2.44(ii)]{Hirs_PGFF}.
\end{proof}

\begin{lemma}\label{lem6:one orbit}
Let $q\equiv0\pmod3$, $q\ge9$.
\begin{description}
  \item[(i)]
  Two lines $\ell_\mu$ and $\ell_{\mu'}$, $\mu, \mu' \in\F_q^*\setminus\{1\}$, $\mu \neq \mu'$  belong to the same orbit of $G_q$ if and only if $\mu=d^4$,  $\mu' = d^4+d^2+1$, $d$ is such that $1-d^2$ is a square, $d \in\F_q^*\setminus\{\pm1\}$, and also $d \neq \pm \sqrt{-1}$ if $q\equiv1\pmod4$.

  \item[(ii)] All lines $\ell_\mu$ with $\mu$ non-square in $\F_q$ belong to distinct orbits of $G_q$.

  \item[(iii)]
 Let $\mu$ by a square in $\F_q$.
  If $q\equiv-1\pmod4$ then at most two $\ell_\mu$-lines belong to the same orbit of $G_q$; there are $(q-3)/8$ pairs of $\ell_\mu$-lines belonging to the same orbit.
  If  $q\equiv1\pmod4$ then at most three $\ell_\mu$-lines belong to the same orbit of $G_q$; there are $(q-9)/8$ pairs of $\ell_\mu$-lines belonging to the same orbit.
  \end{description}
\end{lemma}

\begin{proof}
\begin{description}
  \item[(i)]
  Let $\psi\in G_q$, $\ell_\mu \psi = \ell_{\mu'}$, $\mu \neq \mu'$, let $\MM(\psi)$ be a matrix corresponding to $\psi$. We find it as is a version of $\MM$ of \eqref{eq2_M}. Note that $R_{\mu,\infty}=R_{{\mu'},\infty}=\Pf(1,0,1,0)$ and $R_{\mu,\infty}\psi$  is given by \eqref{eq:infM}. By Lemma \ref{lem:infTOinf_q=0},
if $R_{\mu,\infty} \psi = R_{{\mu'},\infty}$,
 then $R_{\mu,0} \psi = R_{\mu,0}$ that implies $\ell_\mu \psi = \ell_{\mu'}$.
\begin{description}
  \item[(a)] Let $\Pf(1,0,1,0) \psi = \Pf(0,\mu',0,1)$. Then $a=d=0$,  $b, c \neq 0$. As $\MM$ is defined up to a factor of proportionality, we can put $c=1$, obtaining $\mu'=b^2$. On the other hand, $\Pf(0,\mu,0,1)\psi = \Pf(b^3,0,\mu b,0)$ that implies $\mu = b^2$, so $\mu =\mu'$, contradiction.
  \item[(b)] Let $\Pf(1,0,1,0) \psi = \Pf(\zeta,\mu',\zeta,1)$, $\zeta \in \F^*$. Then $a, c \neq 0$ and putting again $c=1$ we obtain
\begin{align}
   & a^3-a-ad^2+bd=0;\db\label{eq6:equiv_q=0_eq1}\\
   &  \mu' = a^2+b^2-abd. \label{eq6:equiv_q=0_eq2}
\end{align}
\begin{description}
  \item[(b1)] Let $\Pf(0,\mu,0,1) \psi = \Pf(1,0,1,0)$. By  \eqref{eq:0xM} with $c=1$ and $a\ne0$, we consistently obtain $d=0$, $\mu ab=0$,  $b = 0$; so $ad-bc=0$, contradiction.
  \item[(b2)] Let $\Pf(0,\mu,0,1) \psi = \Pf(0,\mu',0,1)$. By  \eqref{eq:0xM} with $c=1$ and $a\ne0$, we consistently obtain $b=0$, $\mu ad=0$,  $d=0$; so $ad-bc=0$, contradiction.
  \item[(b3)] Let $\Pf(0,\mu,0,1) \psi = \Pf(\zeta',\mu',\zeta',1)$, $\zeta' \in \F^*, \zeta' \neq \zeta$. By  \eqref{eq:0xM} with $c=1$ and \eqref{eq6:equiv_q=0_eq2}, we have
      \begin{align}\label{eq6:equiv_q=0_eq3}
&\mu(b-ad)+bd^2-b^3=0;\db\\
&\label{eq6:3equiv_q=0_eq1_0}
\mu(a^2d-ab)+b^2d-\mu' d^3= (ad -b)(\mu a -ad^2 + bd^3 -bd)=0.
\end{align}
As $c=1$, $a\neq0$, from $ad-bc \neq 0$ and \eqref{eq6:3equiv_q=0_eq1_0}  we obtain
\begin{equation}\label{eq6:equiv_q=0_eq4}
\mu = (ad^2 - bd^3 +bd)/a.
\end{equation}
Now consider  $R_{\mu,\gamma}\psi = \Pf(\gamma,\mu,\gamma,1)\psi$,  $\gamma \in \F^*$. By \eqref{eq:M_q=0},
\begin{align}\label{eq6:equiv_q=0_im_point}
& [\gamma,\mu,\gamma,1]\MM=
[\gamma a^3+ b^3,~
\mu a^2 d -\mu ab + \gamma a^2 -\gamma abd +\gamma b^2 + b^2 d,\db\\
&-\mu ad +\mu b + \gamma a+ \gamma a d^2 -\gamma bd + b d^2,~
\gamma + d^3].\notag
\end{align}
As $R_{\mu,\infty}\psi, R_{\mu,0}\psi \neq R_{\mu',\infty}$, there exists $ \gamma' \in \F^*$ such that $R_{\mu,\gamma'}\psi =  R_{\mu',\infty}$.
By \eqref{eq6:equiv_q=0_im_point} we obtain $\gamma' = -d^3$ and, using also \eqref{eq6:equiv_q=0_eq4}, $(ad -b)^2(a^2d -ab -d^3 + d)=0$ from which, as $ad-b \neq 0$ and $a\neq 0$, it follows
\begin{equation}\label{eq6:equiv_q=0_eq5}
b = (a^2d -d^3 + d)/a.
\end{equation}
From  \eqref{eq6:equiv_q=0_eq1}, substituting the value of $b$, we obtain
$(a + d)(a - d)(a^2 + d^2 -1)=0$. If $a = \pm d$ then, by \eqref{eq6:equiv_q=0_eq5}, $b = \pm 1$, that implies, by \eqref{eq6:equiv_q=0_eq4},
$\mu = 1$, contradiction. Therefore
\begin{align}\label{eq6:equiv_q=0_eq6}
&a^2 + d^2 -1=0.
\end{align}
This implies $1-d^2$ is a square and from \eqref{eq6:equiv_q=0_eq5}
\begin{equation}\label{eq6:equiv_q=0_eq5new}
b = -ad.
\end{equation}
From \eqref{eq6:equiv_q=0_eq4}, \eqref{eq6:equiv_q=0_eq2}, \eqref{eq6:equiv_q=0_eq6}, and \eqref{eq6:equiv_q=0_eq5new}, we obtain
\begin{equation}\label{eq6:equiv_q=0_eq7}
\mu = d^4, \; \mu' = d^4+d^2+1 = (d+1)^2(d-1)^2.
\end{equation}
For $c=1$, by  \eqref{eq6:equiv_q=0_eq5new},  $ad-bc=-ad \neq 0$.
\end{description}
\end{description}
Now we consider forbidden values of $d$.
As $\mu \neq 1$, $d^4 \neq 1$, so
$d \ne \pm1$, $d \neq \pm \sqrt{-1}$. By \cite[Section 1.5 (ix)]{Hirs_PGFF} and \cite[Section 1.5 (x)]{Hirs_PGFF}, $-1$ is a square if and only if $q\equiv1\pmod4$, whereas $-1$ is not a square if and only if $q\equiv -1 \pmod4$. Also $d \neq 0$, otherwise $\mu=0$, contradiction.
Note that $\mu' =0$ implies $d=\pm1 $, whereas  $\mu' =1$ implies $d=0$ or $d=\pm\sqrt{-1}$, so no other condition on $d$ is obtained considering $\mu'$.

  \item[(ii)] The assertion follows from the case (i).

  \item[(iii)]
  If $\ell_\mu$ and $\ell_{\mu'}$ belong to the same orbit, then there exists a projectivity $\psi' \in G_q$  other than $\psi$ of (i) such that $\ell_{\mu'}\psi' =\ell_{\mu}$. Repeating for  $\psi'$ the same argument as in the case (i) and taking in mind that $\sqrt{-1}$ does not exist for $q\equiv-1\pmod3$, we obtain
\begin{equation}\label{eq6:equiv_q=0_eq8}
\mu' = d'^4, \; \mu = d'^4+d'^2+1 = (d'+1)^2(d'-1)^2, \; d' \in\F_q^*\setminus\{\pm1, \pm\sqrt{-1}\}.\\
\end{equation}
From \eqref{eq6:equiv_q=0_eq7}, \eqref{eq6:equiv_q=0_eq8}  we have 
$\mu +\mu' = d^4+d'^4 = d'^4+d'^2+1 + d^4+d^2+1$ that implies
\begin{equation}\label{eq6:equiv_q=0_eq11}
d^2+d'^2-1=0 , \; d,d' \in\F_q^*\setminus\{\pm1, \pm\sqrt{-1}\}.
\end{equation}

We can see \eqref{eq6:equiv_q=0_eq11}  as the affine equation of the  irreducible conic $d^2+d'^2-z^2=0$, with respect to the infinite line $z=0$.

If  $q\equiv-1\pmod4$, $-1$ is not a square, so the line $z=0$ is external to the conic, as $d^2 \neq -d'^2$ if $d, d' \neq 0$. Therefore all the points of the conic are affine points, so there exist $q+1$ pairs $(d,d')$ satisfying  \eqref{eq6:equiv_q=0_eq11}. However not all pairs are feasible: by the constraints on $d$, the 4 pairs: $(0,\pm1)$, $(\pm1,0)$ must be excluded.

If  $q\equiv 1\pmod4$, $-1$ is a square, so the line $z=0$ has two common points with the conic, namely $\Pf(1,\pm\sqrt{-1},0)$. Therefore there exist $q-1$ pairs $(d,d')$ satisfying  \eqref{eq6:equiv_q=0_eq11}. However not all pairs are feasible: by the constraints on $d$, the 4 pairs $(0,\pm1)$, $(\pm1,0)$  and the 4 pairs $(\pm\sqrt{-1},\pm\sqrt{-1})$,  $(\pm\sqrt{-1},\mp\sqrt{-1})$  must be excluded.

By the symmetry of the equation, if the pair $(d_1, d_2)$ satisfies \eqref{eq6:equiv_q=0_eq11},  the eight pairs $\{(\pm d_i, \pm d_j), (\pm d_i, \mp d_j), i,j \in \{1,2\}, i \ne j\}$ satisfy the same equation. All the eight pairs produce the same unordered pair of values $\{\mu_1=d_1^4, \mu_2=d_2^4\}$, so in total we have $(q-3)/8$ different  pairs of lines $\ell_\mu, \ell_{\mu'}$ belonging to the same orbit for  $q\equiv -1\pmod4$, whereas for  $q\equiv 1\pmod4$ we have $(q-9)/8$ different  pairs of lines $\ell_\mu, \ell_{\mu'}$ belonging to the same orbit.

However it can happen that more than 2 lines of type $\ell_\mu$ belong to the same orbit. If the lines $\ell_\mu, \ell_{\mu'}, \ell_{\mu''}$, $\mu, \mu', \mu''$ pairwise distinct, belong to the same orbit, then there exist $d', d''$ such that
$\mu=d'^4=d''^4$, $d'^4+d'^2+1=\mu'$, $d''^4+d''^2+1=\mu''$.
As both $d'^2$ and $d''^2$ are square roots of $\mu$,  $d'^2 = - d''^2$, otherwise  $\mu' = \mu''$. This implies $-1$ is a square that, by \cite[Section 1.5 (ix)]{Hirs_PGFF}, is equivalent to say $q\equiv1\pmod4$.

If there exists $\mu'''$ belonging to the same orbit $\ell_\mu, \ell_{\mu'}, \ell_{\mu''}$ belong to, then there exists  $d'''$ such that
$\mu=d'''^4$, $d'''^4+d'''^2+1=\mu'''$. As $d'''^2$ is a square root of $\mu$,
$d'''^2= d'^2$ or $d'''^2= d''^2$, so  $\mu'''=\mu'$ or $\mu'''=\mu''$.
Therefore at most three lines of type $\ell_{\mu}$ can belong to the same orbit. \qedhere
\end{description}
\end{proof}

\begin{theorem}\label{th6:equiv_l_muq=0}
Let $q\equiv0\pmod3$, $q\ge9$. Let $\mathfrak{n}_q$ be the total number of distinct orbits of $\ell_\mu$-lines. Let $\mathfrak{S}_q$ be the total number of $\EnG$-lines contained in these orbits. Recall that the total number of $\EnG$-lines is $\#\OO_6=\#\EnG=(q^2-q)(q^2-1)$.
\begin{description}
    \item[(i)] Let $q\equiv-1\pmod4$.  We have $(q-1)/2$ orbits of length $(q^3-q)/2$ generated by $\ell_\mu$-lines, where $\mu$ is a non-square in $\F_q$, and
$(3q-9)/8$ orbits of length $(q^3-q)/4$ generated by $\ell_\mu$-lines, where $\mu$ is a square. This implies $\mathfrak{n}_q=(7q-13)/8$, $\mathfrak{S}_q=(q^3-q)(11q-17)/32\thickapprox\frac{11}{32}\#\OO_{\EnG}\thickapprox0.343\#\OO_{\EnG}$

  \item[(ii)]
 Let $q\equiv1\pmod4$.
  We have $(q-1)/2$ orbits of length $(q^3-q)/2$ generated by $\ell_\mu$-lines, where $\mu$ is a non-square in $\F_q$, and $(3q-3)/8+2\mathfrak{t}_q$ orbits of length $(q^3-q)/4$ generated by $\ell_\mu$-lines, where $\mu$ is a square, $\mathfrak{t}_q$ is the
  number of triples of $\ell_\mu$-lines belonging to the same orbit, $0\le\mathfrak{t}_q\le (q-9)/24$.  This implies $(7q-7)/8\le \mathfrak{n}_q\le(23q-39)/24$, $(q^3-q)(11q-7)/32\le\mathfrak{S}_q\le (q^3-q)(35q-45)/96,~0.343\#\OO_{\EnG}\thickapprox\frac{11}{32}\#\OO_{\EnG}\lesssim\mathfrak{S}_q\lesssim\frac{35}{96}\#\OO_{\EnG}.$

\end{description}
\end{theorem}

\begin{proof}
The assertions follow from Theorem \ref{th:stab_l_mu_q=0} and Lemma \ref{lem6:one orbit} and obvious direct computations. In the point (ii),
 the left inequalities correspond to the case when there are no triples of  $\ell_\mu$-lines belonging to the same orbit, the right ones correspond to the case when all equivalent $\ell_\mu$-lines come in triples.
\end{proof}

\begin{example}\label{examp6}
  Using Magma, for $q\equiv1\pmod4$, we obtain the following:
  \begin{align}\label{eq6:example1}
   q=3^{2m}\equiv1\pmod4,~m=1,\ldots5, ~\mathfrak{t}_q=\frac{q-(-1)^m\sqrt{q}-15}{48}.
  \end{align}
  This implies that for $q=3^{2m},~m=1,\ldots5$, we have exactly $(10q-24-(-1)^m\sqrt{q})/24$ orbits of length $(q^3-q)/4$ generated by $\ell_\mu$-lines, where $\mu$ is a square. Also, $\mathfrak{S}_q= (q^3-q)(34q-48-(-1)^m\sqrt{q})/96\thickapprox\frac{17}{48}\#\OO_{\EnG} \thickapprox0.354\#\OO_{\EnG}$.
\end{example}

\begin{conjecture}
  The results of Example \emph{\ref{examp6}} hold for all $m\ge1$.
\end{conjecture}

\section{On the orbits of lines $\ell_\mu$, $\mu\in\F_q^*\setminus\{1,1/9\}$, $q$ odd, $q\not \equiv0\pmod3$}\label{sec:orbitlinesLmuOdd}

\begin{lemma}\label{lem7:infTOgamqodd}
Let $q$ be odd, $q \not \equiv0\pmod3$. Let  $\psi \in G_q^{\ell_\mu}$ be such that $ R_{\mu,\infty}\psi=R_{\mu,\gamma}$, $\gamma\in\F_q$. If $\mu$ is not a square in $\F_q$ then $\psi$ does not exist. If $\mu$ is a square in $\F_q$ and either $\mu \neq -1/3$, or  $q \not\equiv1\pmod {12}$, or $-1/3$ is not a fourth power,  then a matrix $\MM(\psi)$ corresponding to $\psi$ is as follows.
 \begin{equation}\label{eq7:M_Rinf gamqodd}
   \MM(\psi)=
   \left[
 \begin{array}{cccc}
 0&0&0&1\\
 0&0&b&0\\
 0&b^2&0&0\\
 b^3&0&0&0
 \end{array}
  \right],~b\in\{\sqrt{\mu},-\sqrt{\mu}\}.
\end{equation}
Moreover, for any point $P\in\PG(3,q)$, we have $P(\MM(\psi))^2=P$. Also,
\begin{equation*}
R_{\mu,\infty}\MM(\psi)=R_{\mu,0},~R_{\mu,0}\MM(\psi)=R_{\mu,\infty},~R_{\mu,\gamma}\MM(\psi)=R_{\mu,\gamma'},~\gamma,\gamma'\in\F_q^*,~\gamma'=
\frac{\pm\sqrt{\mu^3}}{\gamma}.
\end{equation*}
If $\mu = -1/3$ and  $q \equiv1\pmod {12}$ and $-1/3$ is a fourth power,  then there are four distinct fourth roots of $-1/3$ and a matrix $\MM'$ corresponding to $\psi$ has either the form \eqref{eq7:M_Rinf gamqodd} or the form \eqref{eq2_M} with
\begin{equation}\label{eq7:M_Rinf gamqodd_1_3}
  a =-b/d,   ~ b =\pm \sqrt{1/3},  ~ c = 1,   ~ d = \t{a fourth root of } -1/3.
\end{equation}
If $\MM'$ has the form \eqref{eq7:M_Rinf gamqodd} it has order $2$, otherwise it has order $3$.
\end{lemma}

\begin{proof}
We consider $\mu\in\F_q^*\setminus\{1,1/9\}$ due to Lemma \ref{lem4:EnGEllMu}. Let $\MM$ be given by \eqref{eq2_M}. By  hypothesis, there exist $\gamma,\zeta\in\F_q$ such that
$[1,0,1,0]\MM=[\gamma,\mu,\gamma,1]$ and $[\zeta, \mu ,\zeta ,1]\MM=[1,0,1,0]$.
This implies
\begin{align}\label{eq7:Im_mu_infty_1qodd}
&a^3 + 3ab^2 - ac^2 - ad^2 - 2bcd=0;\db\\
&\label{eq7:Im_mu_infty_2qodd}
a^2c + 2abd + b^2c-\mu c^3 - 3 \mu cd^2=0.\db\\
&\label{eq7:Im_mu_infty_3qodd}
3 \mu a^2b - 2 \mu acd -  \mu bc^2 + \zeta a^3  + 3\zeta ab^2  -\zeta ac^2  - \zeta ad^2  + b^3 -  2\zeta bcd  - bd^2=0;\db\\
&\label{eq7:Im_mu_infty_4qodd}
 \mu a^2d + 2 \mu abc + \zeta a^2c  + 2\zeta abd  + \zeta b^2c  + b^2d = 0;\db\\
&\label{eq7:Im_mu_infty_5qodd}
3 \mu c^2d + \zeta c^3  + 3\zeta cd^2  + d^3=0.
\end{align}

If $c = 0$ then, by \eqref{eq7:Im_mu_infty_5qodd},  $d^3=0$, so $ad-bc= 0$, contradiction. Therefore, 
, $c\ne0$. As $\MM$ is defined up to a factor of proportionality, we can put $c=1$.
\begin{description}
  \item[(a)] Let $\zeta = 0$.  By \eqref{eq7:Im_mu_infty_5qodd}, $d(3\mu  + d^2)=0$.
\begin{description}
  \item[(a1)] Let $d=0$. By \eqref{eq7:Im_mu_infty_4qodd}, $2\mu ab=0$.
\begin{description}
  \item[(a11)] Let $b=0$. Then $ad-bc= 0$, contradiction.
  \item[(a12)] Let $b\ne0$. Then $a=0$ and, by \eqref{eq7:Im_mu_infty_2qodd}, $ b^2-\mu=0$ that implies $\mu$ is a square and $b = \pm \sqrt{\mu}$.
  Then for $c=1$, $a=d=\zeta=0$, also \eqref{eq7:Im_mu_infty_1qodd} and \eqref{eq7:Im_mu_infty_3qodd} are satisfied and we obtain the matrix \eqref{eq7:M_Rinf gamqodd}.
\end{description}
  \item[(a2)] Let $d \neq 0$. Then, by \eqref{eq7:Im_mu_infty_5qodd}, $-3\mu$ must be a square and $d^2 = -3\mu$.
Equation \eqref{eq7:Im_mu_infty_4qodd} becomes  $\mu a^2d + 2 \mu ab   + b^2d=0$ that gives
 $a = -3b/d$ or $a=b/d$.
 \begin{description}
   \item[(a21)] Let $a=b/d$. Then $ad-bc= 0$, contradiction.
   \item[(a22)] Let  $a = -3b/d$. From  \eqref{eq7:Im_mu_infty_1qodd} we have
$-9 b(\mu-1) (\mu+3 b^2)/d^3=0.$
\begin{description}
  \item[(a221)] Let $b=0$. Then $a=0$ and  $ad-bc= 0$, contradiction.

  \item[(a222)] Let $b\ne0$. Then $-\mu/3$ must be a square and $b^2 = -\mu / 3$.  It follows that $a^2=1$ as $d^2 = -3\mu$. Also, $ ad=-3b$.
Then  \eqref{eq7:Im_mu_infty_3qodd} becomes $32\mu b/3 =0$. As $q$ is odd, this implies $\mu = 0$, contradiction.
\end{description}
 \end{description}
\end{description}
  \item[(b)] Let $\zeta \ne 0$.  By above, this means
$[0, \mu ,0 ,1]\MM=[\gamma',\mu,\gamma',1],~\gamma'\in\F_q$ where $\gamma' \neq \gamma$ with $[1,0,1,0]\MM=[\gamma,\mu,\gamma,1]$. What has been said entails
\begin{align}\label{eq7:Im_mu_infty_6qodd}
&3 \mu a^2b  + b^3      -2 \mu ad -  \mu b  - bd^2  =0;\db\\
&\label{eq7:Im_mu_infty_7qodd}
\mu a^2d + 2 \mu ab  + b^2d  -    3 \mu^2 d   - \mu d^3    =0.
\end{align}
\begin{description}
  \item[(b1)]
Let $d=0$. By \eqref{eq7:Im_mu_infty_5qodd}, $ \zeta=0$, contradiction.

  \item[(b2)] Let  $d \neq 0$. If $a=0$ then, by \eqref{eq7:Im_mu_infty_1qodd}, $-2bd=0$ that implies $b=0$, and $ad-bc=0$, contradiction. So, we should put $a \neq 0$. Now, if $b=0$ then, by \eqref{eq7:Im_mu_infty_6qodd}, $-2 \mu a d = 0$, contradiction. So, we should take also $b \neq 0$.

        By \eqref{eq7:Im_mu_infty_5qodd},
 $\zeta (1+3 d^2) = -3 \mu d - d^3$.
 \begin{description}
   \item[(b21)] Let $1+3 d^2=0$. Then $-3 \mu d - d^3=0$, $ d^2=-1/3$, $\mu = 1/9$, contradiction.
   \item[(b22)] Let $1+3 d^2\ne0$. Then

   \begin{align}\label{eq7:zeta b21}
    \zeta = (-3 \mu d - d^3)/(1   + 3 d^2). 
   \end{align}

Now, by
\eqref{eq7:Im_mu_infty_4qodd},
$(3\mu d^2 - 2\mu  - d^2)a= bd(2d^2 + 1- 3\mu)$.
If $3\mu d^2 - 2\mu  - d^2=0$, then $2d^2 + 1- 3\mu=0$, $\mu = (2d^2 + 1)/3$, and the coefficient of $a$ becomes $2(d - 1)(d + 1)(d^2 + 1/3)$. If $d = \pm 1$, then  $\mu  =1$, contradiction. As $ d^2 +1/3\ne0$ we have $3\mu d^2 - 2\mu  - d^2\ne0$ and
\begin{equation}\label{eq7:Im_mu_infty_val_a_qodd}
a=\frac{2bd^3 + bd- 3\mu bd}{3\mu d^2 - 2\mu  - d^2}.
\end{equation}
Substituting this value of $a$ in
\eqref{eq7:Im_mu_infty_2qodd} (taking into account $d^2 +1/3\ne0$)
 we obtain 
 \begin{align}\label{eq7:Im_mu_infty_eq02_long}
&   b^2 ( 3 \mu^2 d^2 - 4 \mu^2 -4 \mu d^4 + 6\mu d^2-   d^2)=-\mu  (3\mu (3 d^2 - 2) - d^2)^2;
\end{align}
If the coefficient of $b^2$ is $0$, then also
 $ \mu (3 d^2 - 2) - d^2=0$. If $3 d^2 - 2=0$, then $d^2=0$, contradiction; so  $ \mu  = d^2/(3 d^2 - 2)$. Substituting this value of $\mu$ in the coefficient of $b^2$ we obtain 
$d^2(d-1)(d+1)(d^2 + 1/3)/(d^2-2/3)^2 = 0,$
contradiction, as $d \neq 0$, $d^2 + 1/3 \neq 0$ and $d = \pm1$ implies $\mu = 1$.
So, the coefficient of $b^2$ cannot be $0$, and we obtain 
\begin{equation}\label{eq7:Im_mu_infty_val_b_qodd}
    b^2 =
\frac{-\mu  (3\mu d^2 - 2\mu - d^2)^2}{ 3 \mu^2 d^2 - 4 \mu^2 -4 \mu d^4 + 6\mu d^2-   d^2}.
\end{equation}

Substituting the value of $\zeta$ of \eqref{eq7:zeta b21} in \eqref{eq7:Im_mu_infty_3qodd} we have
\begin{equation}\label{eq:eq_4_2}
 (ad-b) (3\mu a^2 + 9\mu abd - 3\mu d^2 + mu - a^2 d^2 - abd - 3 b^2 d^2 - b^2 + d^4  + d^2)=0. 
\end{equation}
Recalling that $ad-b \neq 0$ and substituting the value of $a$ of \eqref{eq7:Im_mu_infty_val_a_qodd} in the second factor of \eqref{eq:eq_4_2} we obtain
\begin{align}\label{eq7:eq_4_3}
&(\mu (3 d^2 - 1) - d^2(d^2+1))(27 \mu^2 b^2 d^2 + 9\mu^2 d^4 - 12\mu^2 d^2\db\\
&+ 4 \mu^2 - 18 \mu b^2 d^2 - 4 \mu b^2 - 6\mu d^4 + 4 \mu d^2  - 4 b^2 d^4 - b^2 d^2 + d^4)=0.\notag
\end{align}
Suppose $\mu (3 d^2 - 1) - d^2(d^2+1)=0$.
If $3 d^2 -1 = 0$ then also $ d^2 (d^2+1)=0$. As $d \neq 0$ this implies $d^2 = -1$, so $-3-1=-4=0$, contradiction, as q is odd.
Therefore $\mu=   (d^4 + d^2)/(3d^2 - 1).$
Substituting this value of $\mu$ in \eqref{eq7:Im_mu_infty_val_b_qodd} we obtain $b^2 = (d^6 + d^4)/(3 d^2 - 1) = \mu d^2$. As $b \neq 0$, from \eqref{eq7:Im_mu_infty_val_a_qodd} we obtain $ab = (d^5 + d^3)/(3 d^2 - 1)$ and finally $abd-b^2=0$, contradiction.

Substituting the value of $b^2$ \eqref{eq7:Im_mu_infty_val_b_qodd} in the second factor of \eqref{eq7:eq_4_3} we obtain
\begin{equation}\label{eq7:eq_4_5}
 d^2 (\mu - 1) (\mu + 1/3) (\mu - 1/9) (3\mu d^2 - 2\mu - d^2)^ 2=0.\\
\end{equation}
Suppose $3\mu d^2 - 2\mu - d^2=0$. Then, by \eqref{eq7:Im_mu_infty_val_b_qodd}, $b^2 =0$, contradiction. 

Suppose $\mu = - 1/3$.
Then by  \eqref{eq7:Im_mu_infty_val_a_qodd},  \eqref{eq7:Im_mu_infty_val_b_qodd},
\begin{align}\label{eq7:Im_mu_infty_val_zeta_a_b_mu1_3}
&a = (3bd^3+3bd)/(1-3d^2);\\
&b^2 = (3d^2-1)^2/(9d^4-18d^2-3).\notag
\end{align}
Substituting this value of $a$ in  \eqref{eq7:Im_mu_infty_1qodd} we obtain

\begin{equation}
bd (d^2 - 1/3)^ 2 (d^2 + 1/3) (d^4 + 1/3)=0.
\end{equation}
By the constraints above, $d^4 =- 1/3$, so $- 1/3$ must be a fourth power. Hence, by  \eqref{eq7:Im_mu_infty_val_zeta_a_b_mu1_3}, $b^2 = 1/3$, so $1/3$ must be a square.
Then  $a= (d(3bd^3+3bd))/(d(1-3d^2))= (b(-1+3d^2))/(d(1-3d^2))=-b/d$ and $ad-b=-2b \neq 0$.
As $\mu$ is a square and $\mu = - 1/3$, by \cite[Section 1.5(xi)]{Hirs_PGFF}  $q \equiv1\pmod3$, so as $1/3$ is also a square, $-1$ must be a square that,
by \cite[Section 1.5(ix)]{Hirs_PGFF}, happens if and only if  $q  \equiv1\pmod4$. The two conditions implies  $q \equiv1\pmod {12}$.
Then, by \cite[Section 1.5(v)]{Hirs_PGFF}, if $- 1/3$ is a fourth power the equation $d^4 =- 1/3$ has four distinct roots. 

The other factors of \eqref{eq7:eq_4_5} cannot be zero and this completes the proof about the structure of the matrix  $\MM(\psi)$.\\
The remaining assertions follow by direct computation using Magma.\qedhere
 \end{description}
 \end{description}
\end{description}
\end{proof}

\begin{remark}
  Note that in Lemma \ref{lem7:infTOgamqodd} the condition 
$q \equiv1\pmod {12}$ and $-1/3$ is a fourth power
is not empty; the examples are  $q = 37, 49, 61,121, 157, 169, 193, 313, \ldots$ 
\end{remark}

\begin{theorem}\label{th7:stab_l_mu_qodd}
Let $q \not \equiv0\pmod3$. Let the $\EnG$-line $\ell_\mu$ be as in \eqref{eq2:ellmu} with $\mu\in\F_q^*\setminus\{1,1/9\}$. Let $G_q^{\ell_\mu}$ be the subgroup  of $G_q$ fixing  $\ell_\mu$ and let $\Os_\mu$ be the orbit of $\ell_\mu$ under $G_q$. Then 
$\#G_q^{\ell_\mu}$, $\#\Os_\mu$  and the structure of $G_q^{\ell_\mu}$ are as follows:
\begin{description}
  \item[(i) ] Let $\mu$ be a non-square in $\F_q$. Then  $\#G_q^{\ell_\mu}=2$, $\#\Os_\mu=(q^3-q)/2$. The
elements of $G_q^{\ell_\mu}$ have the matrix form \eqref{eq:M_Rinfodd}.

  \item[(ii)] Let $\mu$ be a square in $\F_q$ and  either $\mu \neq -1/3$, or  $q \not\equiv1\pmod {12}$, or $-1/3$ is not a fourth power. Then $G_q^{\ell_\mu}$
  is isomorphic to $C_2 \times C_2$ and $\#\Os_\mu=(q^3-q)/4$. Two
elements of $G_q^{\ell_\mu}$ have the matrix form \eqref{eq:M_Rinfodd} and the other two ones are given by \eqref{eq7:M_Rinf gamqodd}.

  \item[(iii)]
  If $\mu = -1/3$ and  $q \equiv1\pmod {12}$ and $-1/3$ is a fourth power, then  $\mu$ is a square in $\F_q$, $G_q^{\ell_\mu}$
  is isomorphic to $A_4$ and $\#\Os_\mu=(q^3-q)/12$. Two
elements of $G_q^{\ell_\mu}$ have the matrix form \eqref{eq:M_Rinfodd},   two other ones are given by \eqref{eq7:M_Rinf gamqodd}, and the last eight elements have the matrix form \eqref{eq7:M_Rinf gamqodd_1_3}.
\end{description}
\end{theorem}

\begin{proof}
  The assertions follow from Lemmas \ref{lemS5:P0notfixedodd} , \ref{lem7:infTOgamqodd}, and \cite[Lemma 2.44(ii)]{Hirs_PGFF}. 
Finally, the only group of order $12$ having  three elements of order two and eight elements of order three is $A_4$, see \cite{groupbook}.
\end{proof}

\begin{lemma}\label{lem7:l0inf_equiv_l_mu}
Let $q$ be odd,  $q\equiv  -1 \pmod3$, $\mu$ be not a square. The line $\Lc$ and a line $\ell_\mu$ belong to the same orbit of $G_q$ if and only if $\mu= -1/3$,  $q\equiv  -1 \pmod{12}$.
\end{lemma}

\begin{proof}
Let $\psi\in G_q$, $\Lc \psi = \ell_\mu$ and let $\MM(\psi)$ be a matrix corresponding to $\psi$. We find it as is a version of $\MM$ of \eqref{eq2_M}. By \eqref{eq2_M}, we have
\begin{equation}\label{eq7:inf_eq1}
[0,0,1,0] \MM=[3ab^2,b^2c+2abd,ad^2+2bcd,3cd^2].
\end{equation}
\begin{description}
  \item[(a)] Let $\Pf(0,0,1,0) \psi = \Pf(1,0,1,0)$. Then $ad^2+2bcd \neq 0$, $3cd^2=0$, $b^2c+2abd=0$  from which we sequently obtain $d \neq 0$, $c=0$, $ab=0$. But $3ab^2 \neq 0$, contradiction.

  \item[(b)] Let $\Pf(0,0,1,0) \psi = \Pf(0,\mu,0,1)$. Then $b^2c+2abd \neq 0$, $ 3ab^2=0$, $ad^2+2bcd=0$ from which we sequently obtain $b \neq 0$, $a=0$, $cd=0$. But $3cd^2 \neq 0$, contradiction.

  \item[(c)] Let
  $\Pf(0,0,1,0) \psi = \Pf(\zeta,\mu,\zeta,1)$, $\zeta \in \F^*_q$. This implies
\begin{align}\label{eq3:inf_eq2}
&3ab^2 -ad^2-2bcd =0;\db\\
&\label{eq3:inf_eq3}
b^2c+2abd -  3 \mu cd^2=0;\db\\
&\label{eq3:inf_eq4}
3ab^2, 3cd^2 \neq 0.
\end{align}
By \eqref{eq3:inf_eq4}, $a,b,c,d \neq 0$. As $\MM$ is defined up to a factor of proportionality, we can put $d=1$. Then by  \eqref{eq3:inf_eq2}
\begin{equation}\label{eq3:inf_eq5}
c= (3ab^2 -a)/2b.
\end{equation}
Substituting this value of $c$ in \eqref{eq3:inf_eq3} we obtain:
$3a(-3\mu b^2 + \mu + b^4 + b^2)/2b=0$ which implies $\mu(3b^2-1)= b^4 + b^2$. Suppose $3b^2-1=0$; then  $b^4 + b^2=4/9=0$, contradiction as $q$ is odd. Therefore $b^2 \neq 1/3$ and
\begin{equation}\label{eq3:inf_eq6}
\mu =( b^4 + b^2)/(3b^2-1).
\end{equation}
If $b= \pm 1$, then $\mu = 1$, contradiction. Therefore in the following we assume $b \neq \pm 1$. By $d=1$ and \eqref{eq3:inf_eq5} we have:
\begin{align}\label{eq3:inf_eq7}
 & [1,0,0,1] \MM =
[8a^3b^3 + 8b^6,~12a^3b^4 - 4a^3b^2 + 8b^5,\db\\
&18a^3b^5 - 12a^3b^3 + 2a^3b + 8b^4,~
27a^3b^6 - 27a^3b^4 + 9a^3b^2 - a^3 + 8b^3].\notag
\end{align}
\begin{description}
  \item[(c1)] Let $\Pf(1,0,0,1) \psi = \Pf(1,0,1,0)$. Then, by \eqref{eq3:inf_eq7},
\begin{align}\label{eq3:inf_p2_eq1}
&2b(b-1)(b+1)(9a^3b^2 - a^3 - 4b^3)=0;\db\\
&\label{eq3:inf_p2_eq2}
 4b^2(3a^3b^2 - a^3 + 2b^3)=0;\db\\
&\label{eq3:inf_p2_eq3}
27a^3b^6 - 27a^3b^4 + 9a^3b^2 - a^3 + 8b^3= 0.
\end{align}
From \eqref{eq3:inf_p2_eq1}, $a^3(9b^2 - 1)= 4b^3$. If $9b^2 - 1=0$ then
$4b^3 =0$, contradiction as $q$ is odd and $b \neq 0$, so $a^3= 4b^3/(9b^2 - 1)$.
From \eqref{eq3:inf_p2_eq2}, substituting the value of $a^3$, we obtain
 $6b^3(5b^2 - 1)/(9b^2 - 1)=0$. If $q\equiv  0 \pmod5$ then $-1=0$, contradiction. Suppose $q \not \equiv  0 \pmod5$; then we obtain
  $b^2 = 1/5.$ Then $a^3=b$. Substituting the values of $a^3$ and $b^2$ in \eqref{eq3:inf_p2_eq3}, we obtain:
 $192b/125 =0$. As $192=2^6 3$, this implies $b=0$, contradiction.

  \item[(c2)] Let $\Pf(1,0,0,1) \psi = \Pf(\zeta',\mu,\zeta',1)$, $\zeta' \in \F, \zeta' \neq \zeta$. This implies
\begin{align}\label{eq3:inf_eq8}
&2b(b-1)(b+1)(9a^3b^2 - a^3 - 4b^3)=0;\db\\
&\label{eq3:inf_eq9}
-27a^3b^{10} + 54a^3b^6 - 32a^3b^4 + 5a^3b^2 + 16b^7 - 16b^5 =0;
\end{align}
Equation \eqref{eq3:inf_eq8} is the same as \eqref{eq3:inf_p2_eq1}, so reasoning as above we obtain $b^2 \neq 1/9$ and $a^3= 4b^3/(9b^2 - 1)$.
Substituting the value of $a^3$ in \eqref{eq3:inf_eq9}, we obtain $b^5(b-1)^2(b+1)^2(3b^4 + 6b^2 - 1)$, that implies
\begin{equation}\label{eq3:inf_eq10}
3b^4 + 6b^2 - 1=0.
\end{equation}
The excluded values $b = 0, \pm 1, 1/3$ and $b^2 =1/3, 1/9$ are not roots of \eqref{eq3:inf_eq10}. Equation \eqref{eq3:inf_eq10} is equianharmonic, as its $I$ and $\Delta$ invariants are $0$ and not $0$, respectively; see \cite[Section 1.11]{Hirs_PGFF}. Therefore, by \cite[Theorem 1.42]{Hirs_PGFF}, as  $q\equiv  -1 \pmod3$, equation \eqref{eq3:inf_eq10} has $0$ or $2$ roots.

Let $y=b^2$. The equation $3y^2 + 6y - 1=0$ has solution in $\F_q$ if and only if $6^2+12=16 \cdot 3$ is a square in $\F_q$. By \cite[Section 1.5]{Hirs_PGFF}, if $q$ is odd,  $q\equiv  -1 \pmod3$, $3$ is a square if and only if  $q\equiv  -1 \pmod{12}$. In this case, we obtain $b^2=(-3 \pm 2 \sqrt{3})/3$.

Exactly one of the two values of $b^2$ is a square in $\F_q$. In fact, if both the values are squares or not squares, $[(-3 - 2 \sqrt{3})/3 ] \cdot [(-3 + 2 \sqrt{3})/3 ]=-1/3$ would be a square, contradiction as $3$ is a square and $-1$ is not a square.
Therefore, equation \eqref{eq3:inf_eq10} has exactly $2$ solutions in $\F_q$.
Anyway, substituting both the possible values of $b^2$ in \eqref{eq3:inf_eq6},
we obtain the same value $\mu = -1/3$.

If $ad-bc=0$, then $a = (3ab^2-a)/2$. As $a  \neq 0$, this implies $b^2=1$, contradiction as $b \neq \pm1$. \qedhere
\end{description}
\end{description}
\end{proof}

\begin{lemma}\label{lem7:l0inf_equiv_l_mu_2}
Let $q\equiv  1 \pmod{12}$, $\mu= -1/3$, $-1/2$ be a cube and $-1/3$ be a fourth power. Then $\mu$ is a square and the lines $\Lc$ and $\ell_{-1/3}$ belong to the same orbit of $G_q$.
\end{lemma}

\begin{proof}
Let $\psi\in G_q$, $\Lc \psi = \ell_{-1/3}$ and let $\MM(\psi)$ be a matrix corresponding to $\psi$.  We find it as is a version of $\MM$ of \eqref{eq2_M}. By \eqref{eq2_M} we have \eqref{eq7:inf_eq1}.
Reasoning  as in Lemma~\ref{lem7:l0inf_equiv_l_mu}, we obtain $\Pf(0,0,1,0) \psi = \Pf(\zeta,-1/3,\zeta,1), \zeta \in \F^* $,
that implies
\begin{align}\label{eq7:inf_eq2}
&3ab^2 -ad^2-2bcd =0;\db\\
&\label{eq7:inf_eq3}b^2c+2abd + cd^2=0;\db\\
&\label{eq7:inf_eq4}3ab^2, 3cd^2 \neq 0.
\end{align}
By \eqref{eq7:inf_eq4}, $a,b,c,d \neq 0$. As $\MM$ is defined up to a factor of proportionality, we can put $d=1$. Then by  \eqref{eq7:inf_eq2},
$c= (3ab^2 -a)/2b.$
Substituting this value of $c$ in \eqref{eq7:inf_eq3} we obtain 
$a( 3b^4 + 6 b^2-1)=0$ which implies
\begin{equation}\label{eq7:inf_eq6}
3b^4 + 6 b^2-1=0.
\end{equation}
As $q$ is odd, the values $\pm 1, \pm 1/3$, are not roots of the equation \eqref{eq7:inf_eq6}, so in the following we assume $b \neq \pm 1, \pm 1/3$.

Reasoning  as in Lemma \ref{lem7:l0inf_equiv_l_mu}, we obtain
 $\Pf(1,0,0,1) \psi = \Pf(\zeta',\mu,\zeta',1)$, $\zeta' \in \F, \zeta' \neq \zeta$, and
 \begin{equation}\label{eq7:value_a_3}
a^3= 4b^3/(9b^2 - 1).
\end{equation}
 As $q\equiv  1 \pmod{12}$,  $-3$ and $-1$ are squares in $\F_q$ by \cite[Section 1.5 ]{Hirs_PGFF} . Therefore also $3$ is a square.
 Let $\alpha$ be a fourth root of $-1/3$ and 
 $\beta$ be a square root of $3$ and
 $\gamma$ be a cubic root of $-1/2$. Then Equation \eqref{eq7:inf_eq6} can be factorized in the following way:
\begin{equation*}
\begin{gathered}
 (2b + (-\alpha^3 - \alpha)\beta  -3\alpha^3 - \alpha) (2b + (\alpha^3 - \alpha)\beta + 3\alpha^3 + \alpha)\\
 (2b + (\alpha^3 + \alpha)\beta -3\alpha^3 - \alpha) (2b + (\alpha^3 + \alpha)\beta + 3\alpha^3 + \alpha)=0.
\end{gathered}
\end{equation*}
Choosing, for example, 
\begin{equation*}
b=((\alpha^3 + \alpha) \beta + 3 \alpha^3 +  \alpha)/2,
\end{equation*}
Equation \eqref{eq7:value_a_3} can be factorized in the following way:
\begin{equation*}
\begin{gathered}  
    (2a + (\gamma \alpha^3 - \gamma \alpha) \beta + 3 \gamma \alpha^3 + \gamma \alpha) (2a + (\gamma\alpha^3+ \gamma\alpha) \beta  -3\gamma \alpha^3 + \gamma\alpha)\\
     (a - \gamma \alpha^3 \beta - \gamma \alpha)=0
\end{gathered}
\end{equation*}
and we can choose, for example, $a = \gamma \alpha^3 \beta + \gamma \alpha$. Note that $a \neq 0$. In fact $a = \gamma \alpha(\alpha^2 \beta + 1)=0 $ implies $\alpha^2 \beta =-1$ whence $(\alpha^2 \beta)^2 =(-1)^2$ which simplifies to $-1=1$, contradiction, as $q$ is odd.\\
Finally, if $ad-bc=0$, then $a = (3ab^2-a)/2$. As $a  \neq 0$, this implies $b^2=1$, contradiction as $b \neq \pm1$. 

\end{proof}

\begin{remark}
 Note that in Lemma \ref{lem7:l0inf_equiv_l_mu_2} the condition 
  $q\equiv  1 \pmod{12}$, $-1/2$ is a cube, and $-1/3$ is a fourth power
is not empty; the examples are $ q = 121, 157, 397, 433, 529, 601, 625$, $961, 997, \ldots$
For $q = 121$ we actually tested by Magma that $\Lc$ and $\ell_{-1/3}$ give the same orbit. 
\end{remark} 
\begin{theorem}\label{th7:l0inf_equiv_l_mu}
The line $\Lc$ and a line $\ell_\mu$ belong to the same orbit of $G_q$ if and only if $q$ is odd and $\mu= -1/3$,  $q\equiv  -1 \pmod{12}$ and $\mu$ is not a square, or  $q\equiv  1 \pmod{12}$ and $\mu$ is a square  and $1/2$ is a cube and $-1/3$ is a fourth power.
\end{theorem}

\begin{proof}
If $q$ is even,
or $q$ is odd and $q\equiv  1 \pmod3$ and $-1/2$ is not a cube,
or $\mu$ is a square and $\mu \neq -1/3$,
or $\mu$ is a square and $\mu = -1/3$ and  $q \not\equiv  1 \pmod {12}$,
or $\mu$ is a square and $\mu = -1/3$ and  $q \equiv  1 \pmod {12}$ and $-1/3$ is not a fourth power,
by  Theorems \eqref{th3:main}, \eqref{S5:stab} and \eqref{th7:stab_l_mu_qodd}, the line $\Lc$ and a line $\ell_\mu$ have different stabilizer groups, so they cannot belong to the same orbit. 

 Then apply Lemmas \eqref{lem7:l0inf_equiv_l_mu} and \eqref{lem7:l0inf_equiv_l_mu_2}.
\end{proof}

\end{document}